\documentclass{article}
\usepackage{amsmath,amsfonts,amssymb,dsfont,mathrsfs}
\usepackage{amsthm}
\usepackage[dvipsnames]{xcolor}
\usepackage{caption}
\usepackage{url}
\usepackage{adjustbox}
\usepackage{lscape}
\usepackage{rotating}

\usepackage{tikz}
\usepackage{tikz-3dplot}
\usetikzlibrary{calc,,angles,quotes}

\usepackage{subcaption}

\usepackage[linesnumbered,ruled]{algorithm2e}
\SetKwComment{Comment}{/* }{ */}
\SetKwInOut{Input}{Input}
\SetKwInOut{Output}{Output}
\providecommand{\U}[1]{\protect\rule{.1in}{.1in}}

\newtheorem{theorem}{Theorem}[section]
\newtheorem{corollary}[theorem]{Corollary}
\newtheorem{lemma}[theorem]{Lemma}
\newtheorem{proposition}[theorem]{Proposition}
\newtheorem{definition}[theorem]{Definition}
\newtheorem{remark}[theorem]{Remark}
\theoremstyle{definition}
\newtheorem{example}[theorem]{Example}

\let\epsilon\varepsilon

\providecommand{\tto}{\mathop{\rightrightarrows}\nolimits}

\newcommand{\traza}[2]{#1\big|_{#2}}

\newcommand{\R}{\mathbb{R}}

\newcommand{\N}{\mathbb{N}}

\renewcommand{\P}{\mathbb{P}}
\newcommand{\E}{\mathbb{E}}

\newcommand{\ind}{\mathds{1}}

\providecommand{\dom}{\mathop{\rm dom}\nolimits}
\providecommand{\gph}{\mathop{\rm gph}\nolimits}
\providecommand{\argmin}{\mathop{\rm argmin}\nolimits}
\providecommand{\argmax}{\mathop{\rm argmax}\nolimits}

\providecommand{\into}{\mathop{\rm int}\nolimits}
\providecommand{\ri}{\mathop{\rm ri}\nolimits}

\providecommand{\aff}{\mathop{\rm aff}\nolimits} 
\providecommand{\dim}{\mathop{\rm dim}\nolimits} 
\providecommand{\codim}{\mathop{\rm codim}\nolimits} 
\providecommand{\adim}{\mathop{\rm dim}\nolimits} 
\providecommand{\ext}{\mathop{\rm ext}\nolimits}
\providecommand{\ext}{\mathop{\rm ext}\nolimits}

\newcommand{\Sph}{\mathbb{S}}
\newcommand{\nx}{{n_x}}
\newcommand{\ny}{{n_y}}

\definecolor{forestgreenweb}{rgb}{0.13, 0.55, 0.13}

\newcommand{\DS}[1]{#1} 
\newcommand{\AS}[1]{#1} 
\newcommand{\GM}[1]{#1} 

\title{Exploiting the polyhedral geometry of stochastic linear bilevel programming\footnote{A short version of this article was published at the proceedings of IPCO 2023~\cite{MSS2023exploiting}. This extended version contains more detailed discussions, examples, results, and proofs.}}

\author{
	Gonzalo Muñoz \thanks{ Instituto de Ciencias de la Ingenier\'ia, Universidad de O'Higgins, Rancagua, Chile. {\tt gonzalo.munoz@uoh.cl} }
	\and 
	David Salas \thanks{Instituto de Ciencias de la Ingenier\'ia, Universidad de O'Higgins, Rancagua, Chile. {\tt david.salas@uoh.cl}}
	\and
	Anton Svensson \thanks{Instituto de Ciencias de la Ingenier\'ia, Universidad de O'Higgins, Rancagua, Chile. {\tt anton.svensson@uoh.cl}}
}
\date{}
\begin{document}
	
	\maketitle
	
	\begin{abstract}
		   We study linear bilevel programming problems whose lower-level objective is given by a random cost vector with known distribution. We consider the case where this distribution is nonatomic, allowing to reformulate the problem of the leader using the Bayesian approach in the sense of Salas and Svensson (2023), with a \DS{decision-dependent distribution that concentrates on the vertices of the feasible set of the follower's problem. We call this a vertex-supported belief}. 
   We prove that 
   this formulation is piecewise affine over the so-called chamber complex of the feasible set of the high-point relaxation. 
   We propose two algorithmic approaches to solve general problems enjoying this last property. The first one is based on enumerating the vertices of the chamber complex. \GM{This approach is not scalable, but we present it as a computational baseline and for its theoretical interest.}
   The second one is a Monte-Carlo approximation scheme based on the fact that randomly drawn points of the domain lie, with probability 1, in the interior of full-dimensional chambers, where the problem (restricted to this chamber) can be reduced to a linear program.
   Finally, we evaluate these methods through computational experiments showing both approaches' advantages and challenges.
	\end{abstract}
	
\section{Introduction \label{sec:Intro}}

Stackelberg games, also referred to as bilevel
programming problems, were first introduced by H. von Stackelberg in \cite{Stackelberg1934market}. In this seminal work, an economic equilibrium problem between two firms was studied, under the particularity that one of them, the leader, is able to anticipate the decisions of the other one, the follower. 
Bilevel programming is  
an active field of research, and we refer the reader to the monographs \cite{dempe2002foundations,Dempe2015Bilevel} for comprehensive introductions, and to \cite{DempeZemkoho2020Advances} for recent developments.
In the last decade, researchers have started to consider uncertainty in Stackelberg games. A recent survey by Beck, Ljubi\'{c} and Schmidt \cite{beck2023survey} provides an overview of new questions and contributions on this topic.

One model that considers uncertainty in Stackelberg games is the Bayesian approach \cite{MallozziMorgan1996,SalasSvensson2020Existence}. Their starting point is that for any given leader's decision $x$, the leader only knows that the reaction $y$ of the follower is selected from a set $S(x)$, making $y$ a decision-dependent and uncertain 
\AS{variable}. The leader endows the set $S(x)$ with a probability distribution $\beta_x$ which models how the leader believes that the possible responses of the follower are distributed.

Uncertainty in the data of the lower-level has been considered by Claus for linear bilevel programming problems from a variational perspective considering risk measures (see the survey \cite{BurtscheidtClaus2020BilevelLinear} and the references therein, and the recent works \cite{Claus2021Continuity,Claus2022Existence}). In \cite{Ivanov2018RandomParameters}, Ivanov considered the cost function of the follower as a bilinear form $\langle Ax + \xi(\omega),y\rangle$, \DS{but with fixed feasible set for the follower. The method proposed to solve the problem was a Sample Average Approximation scheme, to build a deterministic bilevel program whose size depends on the number of samples.} 
Recently, in \cite{BuchheimHenkeIrmai2022Knapsack}, Buchheim, Henke and Irmai considered a bilevel version of the continuous knapsack problem with uncertainty on the follower's objective.

In this work, we consider a linear bilevel programming problem where the lower-level objective is uncertain for the leader but follows a prior known distribution (as the particular case studied in \cite{BuchheimHenkeIrmai2022Knapsack}). We study the problem from a 
Bayesian perspective \cite{SalasSvensson2020Existence}, and by means of the so-called \emph{chamber complex} of a  polytope (see Section \ref{sec:Chambers}), which subdivides the space of the leader's decisions in a meaningful way. The idea of using the chamber complex to understand geometrical properties of optimization 
problems under uncertainty is not new, but it is recent. To the best of our knowledge, the first work that does this is \cite{forcier2020polyhedral} (see also \cite{forcier2022thesis}), on which multistage stochastic linear optimization is studied. However, the techniques there cannot be 
extended to Stackelberg games, since the latter carries an intrinsical nonconvexity.

 
\subsection{Problem formulation}

Our study focuses on the setting of linear bilevel programming, which is the case where the data (objective functions and constraints) of the problem is linear \DS{and all variables are continuous}. More precisely, we aim to study the problem where the leader decides a vector $x\in\R^{\nx}$ that solves
\begin{equation}\label{eq:target-problem}
	\left\{\begin{array}{cl}
		\displaystyle\min_{x\in\R^{\nx}} & \langle d_1,x\rangle + \E[\langle d_2,y(x,\omega)\rangle]\vspace{0.2cm}\\
		s.t. & \DS{A_l x \leq b_l}\\
		&y(x,\omega)\text{ solves }\left\{\begin{array}{cl}
			\displaystyle\min_{y\in\R^{\ny}} & \langle c(\omega), y\rangle  \\
			s.t. & \begin{array}{l}
				\DS{A_fx + B_fy \leq b_f,} 
			\end{array}
		\end{array}\right.\quad\text{a.s. }\omega\in\Omega,
	\end{array}\right.
\end{equation}
where \DS{$A_l\in \R^{m_l\times \nx}$, $A_f\in \R^{m_f\times \nx}$, $B_f\in \R^{m_f\times \ny}$, $b_l\in \R^{m_l}$, $b_f\in \R^{m_f}$,} $d_1\in \R^{\nx}$, $d_2\in \R^{\ny}$, and $c:\Omega\to \Sph_{\ny}$ is a random vector over a probability space $(\Omega,\Sigma,\mathbb{P})$ with values in the unit sphere of $\R^{\ny}$. \DS{The subindexes $l$ and $f$ stand for leader and follower, respectively}. The notation carries the usual ambiguity of bilevel problems, which appears whenever the lower-level optimal response $y(x,\omega)$ is not uniquely determined for some $x$. However, we focus our attention here on costs whose distributions are nonatomic (see Section \ref{subsec:bayesianformulation}), which implies that with probability 1, $y(x,\omega)$ is unique for all \DS{feasible $x$}.

\DS{
\begin{remark} \label{rmk:coupling}In full generality, 
the upper-level constraint of a deterministic bilevel problem is of the form $A_lx + B_ly \leq b_l$. These constraints are known as \emph{coupling constraints}, and they introduce another level of difficulty to bilevel programming,  for example, by producing disconnected feasible regions (see, e.g., \cite{MershaDempe2006,AudetHaddadSavard2006}).  However, it was recently shown that for any linear bilevel problem with coupling constraints, there is a problem without coupling constraints with the same optimal solutions \cite{henke2024coupling}. For the stochastic setting we are considering, 
 one way to incorporate coupling constraints in Problem \eqref{eq:target-problem} could take the form of
	\begin{equation}\label{eq:CouplingConstraints}
		A_lx+B_ly(x,\omega) \leq b_l,\quad \text{ a.s. }\omega\in\Omega,
	\end{equation}
where coupling constraints must be satisfied for almost every follower's response.  It is unclear if a similar result to \cite{henke2024coupling} holds in the stochastic setting. In any case, in this first study, we prefer to keep constraints of the form \eqref{eq:CouplingConstraints} out of the scope of this work.
\end{remark}
}

\AS{To fix ideas, let us start with a warm-up example.
\begin{example}\label{ex:banana}
Let $\nx=2$ and $\ny=1$. For the follower, consider a random cost $c$ such that $\P(c=1)=\P(c=-1)=\frac{1}{2}$ and the constraints defined by the inequalities \(
x_1-y\leq 0,\; -x_1-y\leq 0,\; x_2+y\leq 1,\text{ and }-x_2+y\leq 1
\). For the leader, let the costs be $d_1=(2,1)$ and $d_2=4$, \DS{and consider no constraints. Then, Problem \eqref{eq:target-problem} is given by
\begin{equation}
	\left\{\begin{array}{cl}
		\displaystyle\min_{x\in\R^{2}} & 2x_1 + x_2 + \E[4y(x,\omega)]\vspace{0.2cm}\\
		s.t. 
		&y(x,\omega)\text{ solves }\left\{\begin{array}{cl}
			\displaystyle\min_{y\in\R} & c(\omega)y  \\
			s.t. & \begin{array}{rl}
				x_1-y&\leq 0,\\
				-x_1-y&\leq 0,\\
				x_2+y&\leq 1,\\
				-x_2+y&\leq 1.
			\end{array}
		\end{array}\right.\quad\text{a.s. }\omega\in\Omega,
	\end{array}\right.
\end{equation}
Observe that, given $x\in\R^2$, the feasible set of the follower can be written as
\begin{equation}
S(x) = \{ y\in\R\ :\ |x_1|\leq y \leq 1-|x_2|  \},
\end{equation}
and thus, the feasible set of the leader is given by all points for which the lower-level admits a feasible point, that is,
\begin{equation}
	X := \{ x\in \R^2\ :\ S(x) \neq \emptyset \} = \{  x\in \R^2\ :\ |x_1|+|x_2|\leq 1  \}.
\end{equation}
Finally, for any $x\in X$,} the solution $y(x,\omega)$ of the lower-level is either equal to $|x_1|$ with probability $\frac{1}{2}$ or equal to $1-|x_2|$ with probability $\frac{1}{2}$. So problem \eqref{eq:target-problem} can be rewritten as
\begin{equation}
    \left\{\begin{array}{cl}
		\displaystyle\min_{x\in\R^{2}} & 2x_1+x_2 + 4\left(\frac{1}{2}|x_1|+\frac{1}{2}(1-|x_2|) \right)\\
		s.t. & |x_1|+|x_2|\leq 1\\
	\end{array}
 \right.
\end{equation}
This problem is nonconvex and piecewise linear, and its optimal solution is $x=(0,1)$.\hfill$\Diamond$
\end{example}
}

\GM{
Problem \eqref{eq:target-problem} is computationally challenging. Its deterministic version is well-known to be strongly NP-Hard \cite{hansen1992new}, and even checking local optimality is NP-Hard \cite{vicente1994descent}. Recently, it was proven that linear bilevel optimization \AS{(both in its optimistic and pessimistic versions)} belongs to NP \cite{buchheim2023bilevel}.

For the stochastic version considered here, we are only aware of the hardness results in
\cite{BuchheimHenkeIrmai2022Knapsack}, who consider the special case of a continuous knapsack problem (see Section \ref{sec:continuousknapsack}); here, the authors prove that if the knapsack item values are independently and uniformly distributed, the resulting problem is \#P-Hard.
To the best of our knowledge, there is no stronger hardness for the general case of problem \eqref{eq:target-problem}.
On the other hand, the robust case (i.e., when the leader takes a worst-case perspective on the follower's cost vector) has even stronger complexity results \cite{buchheim2021complexity,buchheim2022robust}.
}

\subsection{Our contribution}

The contributions of this work can be summarized as follows. First, we reformulate problem \eqref{eq:target-problem} following the Bayesian approach \cite{SalasSvensson2020Existence} using a belief $\beta$ induced by the random cost $c(\omega)$, which we call \emph{vertex-supported belief induced by $c$}. This is done in Section \ref{sec:Pre}, along with some preliminaries. 

Secondly, we show that the objective function of problem \eqref{eq:target-problem} as well as its discretization using sample average approximation methods (see, e.g., \cite{HomemdeMelloBayraksan2014Survey}) have the property of being  \emph{piecewise linear over the chamber complex} of 
\DS{an appropriate polytope $D\subseteq \R^{\nx}\times \R^{\ny}$.}
That is, we show that for every polytope in the chamber complex of $D$, the objectives of the aforementioned problems are affine within that polytope. This is done in Section 3.

\AS{
In Section \ref{sec:Examples} we provide some illustrative examples that highlight some difficulties of working with the chamber complex of a polytope, hidden in their implicit definition.}

Section \ref{sec:Algorithms} and Section \ref{sec:Monte-Carlo} are devoted to propose two methods to solve piecewise linear problems over the chamber complex of a polytope. The first proposal, in Section \ref{sec:Algorithms}, is a deterministic algorithm that is based on straightforward observation: the optimal solution of such a problem must be attained at a vertex of the chamber complex. Thus, we provide a MIP formulation to enumerate all the vertices of the chamber complex. \DS{ The proposed algorithm is not scalable since the \AS{MIPs to be solved are of the size of the Face Lattice of $D$ and the number of vertex chambers could be exponential (see e.g. Example \ref{ex:expchambervertices})}. However, we include it due to its theoretical interest in understanding the combinatorial nature of Problem \eqref{eq:target-problem}. } \GM{It also serves as an exact computational baseline for inexact algorithms.} The second proposal, in Section \ref{sec:Monte-Carlo}, is a Monte-Carlo algorithm that is based on another very simple observation: if a point $x$ is drawn randomly, it will belong to the interior of a full-dimensional chamber with probability 1. \DS{It is worth to mention that either of the two methods rely on samples of the cost function $c(\omega)$, instead of evaluating exactly the objective function of the leader. This yields that the size of samples used to estimate the cost function only affects the complexity of the proposed methods by a linear factor.} 

We finish our work by testing both solution methods over numerical experiments. This is done in Section \ref{sec:Numerical}.

\section{Preliminaries \label{sec:Pre}}

For an integer $n\in\N$, we write $[n]:=\{1,\ldots,n\}$. Throughout this work we consider Euclidean spaces $\R^{n}$ endowed with the 
\AS{Euclidean} norm $\|\cdot\|$ and their inner product $\langle \cdot,\cdot\rangle$. We denote by  
$\Sph_n$  
the unit sphere in $\R^{n}$. 
For a set $A\subseteq \R^{n}$, we will write $\into(A)$, $\overline{A}$, $\aff(A)$, to denote its interior, closure and affine hull, respectively. We denote by $\ri(A)$ its relative interior, which is its interior with respect to the affine space $\aff(A)$. We denote the affine dimension of $A$ as $\adim(A)$, which corresponds to the dimension of $\aff(A)$. It is well known that if $A$ is nonempty and convex, then $\ri(A)$ is nonempty and $\overline{\ri(A)} = \overline{A}$. For a convex set $A$ and a point $x\in A$, we write $N_A(x)$ to denote the normal cone of $A$ at $x$, i.e.,
\begin{equation}\label{eq:NormalCone}
	N_A(x) := \{ d\in \R^{n}\ :\ \langle d,y-x\rangle \leq 0\quad \forall y\in A \}.
\end{equation}
We write $\ind_A$ to denote the indicator function of a set $A\subseteq\R^n$, having value $1$ on $A$ and $0$ elsewhere.

For a function $f:A\subseteq \R^{n}\to \R$ and $\alpha\in \R$, we write $[f\leq \alpha]$, $[f\geq \alpha]$ and $[f = \alpha]$ to denote the $\alpha$-sublevel set, the $\alpha$-superlevel set, and the $\alpha$-level set, respectively. 

A face $F$ of a polyhedron $P$ is the argmax of a linear function over $P$. The collection of all faces of $P$, that we denote by $\mathscr{F}(P)$, is known as the \textit{Face Lattice} of $P$. For any $k\in \{0,\ldots,n\}$ we will write
\begin{align}
	\mathscr{F}_{k}(P) &:= \{ F\in\mathscr{F}(P)\ :\ \adim(F) = k \}\label{eq:FixedDimension-Faces},\\
	\mathscr{F}_{\leq k}(P) &:=\bigcup_{j=0}^k\mathscr{F}_{j}(P) =  \{ F\in\mathscr{F}(P)\ :\ \adim(F) \leq k \}.\label{eq:LowerDimension-Faces}
\end{align}
The zero-dimensional faces are the vertices and we write $\ext(P):=\mathscr{F}_0(P)$.

In what follows, concerning problem \eqref{eq:target-problem}, we use $\nx$ and $\ny$ to denote the number of variables of the leader and the follower, respectively. Similarly, we use $m$ for the number of constraints in the follower's problem. 
Finally, given two nonempty sets $X\subseteq \R^{\nx}$ and $Y\subseteq \R^{\ny}$, we write $M:X \tto Y$ to denote a set-valued map, i.e., a function assigning to each $x\in X$ a (possibly empty) set $M(x)$ in $Y$. The graph of $M$ is the set 
\(
    \gph M=\{(x,y)\,:\, y\in M(x)\},
\)
and the domain of $M$ is the set $\dom M:=\{x\,:\,M(x)\neq\emptyset\}$.

\subsection{Specific notation}
\DS{
Motivated by the structure of problem \eqref{eq:target-problem}, we define $m=m_l+m_f$ and
\[
A = \begin{bmatrix}
	A_l\\ A_f
\end{bmatrix}\in\R^{m\times n_x},\quad B = \begin{bmatrix}
0\\ B_f
\end{bmatrix}\in\R^{m\times n_y},\quad\text{ and }\quad b=\begin{bmatrix}
b_l\\ b_f
\end{bmatrix}\in\R^{m}.
\]
With this, we define the polyhedron $D$ as 
\begin{equation}\label{eq:HighPoint-Relax}
	D := \left\{ (x,y) \in \R^{\nx}\times \R^{\ny}\ :\ Ax + By \leq b   \right\},
\end{equation}
which is the feasible set of the high-point relaxation of Problem \eqref{eq:target-problem} (see, e.g., \cite{kleinert2021survey}). We then consider the \textit{feasible set $X$ of the leader} as all points $x\in\R^{\nx}$ verifying the leader's constraints and admitting at least one feasible point for the follower. That is,
\begin{equation}\label{eq:DomainX}
    X  := \big\{ x\in \R^{\nx}\ :\ \exists y\in \R^{\ny}\text{ such that }Ax + By \leq b \big\}.
\end{equation}
}
It will be assumed along the paper that $D$ is full-dimensional. We do not \DS{lose} generality since it is always possible to embed $D$ into $\R^{\adim(D)}$. 
We will also assume that $D$ is compact, i.e. it is a polytope.
Similarly, we define the ambient space for the follower's decision vector as
	\begin{equation}\label{eq:DomainY}
			Y := \big\{ y\in \R^{\ny}\ :\ \exists x\in \R^{\nx}\text{ such that }Ax + By \leq b \big\}.
		\end{equation}
Both $X$ and $Y$ are polytopes (bounded polyhedra) as $D$ is assumed so.  We write $\pi:\R^{\nx}\times \R^{\ny}\to \R^{\nx}$ to denote the parallel projection given by $\pi(x,y) = x$. In particular, equation \eqref{eq:DomainX} can be written as $X = \pi(D)$. We consider the set-valued map $S:X\tto \R^{\ny}$ defined as 
	\begin{equation}\label{eq:sliceMap}
		S(x) := \{y \in \R^{\ny}\ :\ (x,y) \in D\}.
	\end{equation}
We call $S$ the \emph{fiber map} of $D$ (through the projection $\pi$). Clearly, $S$ has nonempty convex and compact values. 

For each $i\in [m]$ we define the function $g_i: \R^{\nx}\times \R^{\ny} \to \R$ given by
\begin{equation}\label{eq:gi-constraint}
	g_i(x,y) := A_{i}x + B_{i} y - b_i,
\end{equation}
where $A_i$ and $B_i$ denote the $i$th row of $A$ and $B$, respectively. Thus, we can write 
\(
D = \bigcap_{i\in[m]}[g_i\leq 0].
\)
Since every function $g_i$ is affine, we will simply write $\nabla g_i$ to denote their gradients at any point. For any subset $J\subseteq [m]$, we define the function $g_J: \R^{\nx}\times \R^{\ny} \to \R$ given by
\begin{equation}\label{eq:gJ-sumconstraint}
	g_J(x,y) := \sum_{j\in J} g_j(x,y).
\end{equation}
For simplicity we write $\mathscr{F} := \mathscr{F}(D)$ and  for each $n\in\{ 0,1,\ldots,\nx+\ny\}$, $\mathscr{F}_n := \mathscr{F}_n(D)$ and $\mathscr{F}_{\leq n} := \mathscr{F}_{\leq n}(D)$. Note that, for each face $F\in \mathscr{F}_n$ we can select a set $J\subseteq [m]$ of cardinality $|J| = (\nx+\ny) - n$ such that
\begin{equation}\label{eq:FacesToSets}
	F = D\cap [g_J = 0].
\end{equation}
Thus, for any face $F\in \mathscr{F}(D)$, we define $g_F:= g_J$, where $J$ is the (selected) subset of $[m]$ verifying \eqref{eq:FacesToSets}. Observe that since $|J|$ coincides with the codimension of $F$, the set $\{\nabla g_j\ :\ j\in J\}$ must be linearly independent. 

\subsection{Vertex-supported beliefs and Bayesian formulation}
\label{subsec:bayesianformulation}

In what follows, we write $\mathcal{B}(Y)$ to denote the Borel $\sigma$-algebra of $Y$ and $\mathscr{P}(Y)$ to denote the family of all (Borel) probability measures on $Y$.

Recall from \cite{SalasSvensson2020Existence} that for a set-valued map $M:X\tto Y$ with closed and nonempty values, a map $\beta:X\to \mathscr{P}(Y)$ is said to be \emph{a belief over $M$} if for every $x\in X$, the measure $\beta(x) = \beta_x$ concentrates over $M(x)$, i.e., $\beta_x(M(x)) = 1$.  

Let $(\Omega,\Sigma,\P)$ be a probability space. To model the cost function of the lower-level problem in \eqref{eq:target-problem}, we will consider only random vectors $c:\Omega\to \Sph_{\ny}$ with nonatomic distributions, in the sense that
\begin{equation}\label{eq:NonatomicDistribution}
	\forall O\in \mathcal{B}(\Sph_{\ny}):\quad \mathcal{H}^{\ny-1}(O) = 0 \implies \P[c(\omega)\in O] = 0,
\end{equation}
where $\mathcal{H}^{\AS{\ny}-1}$ denotes the $(\ny-1)$-Hausdorff measure over $(\Sph_{\nx}, \mathcal{B}(\Sph_{\ny}))$. In other words, the probability measure $\P\circ c^{-1}$ is absolutely continuous with respect to $\mathcal{H}^{\ny-1}$. Note that with this definition, any random vector $c:\Omega\to \R^{\ny}$ having an absolutely continuous distribution with respect to the Lebesgue measure $\mathcal{L}^{\ny}$ induces an equivalent random vector $\bar{c}:\Omega\to \Sph_{\ny}$ given by $\bar{c}(\omega) := \frac{c(\omega)}{\|c(\omega)\|}$. This new random vector is well-defined almost surely in $\Omega$, except for the negligible set $N = c^{-1}(0)$, and using $c(\cdot)$ or $\bar{c}(\cdot)$ in problem \eqref{eq:target-problem} is equivalent.

Now, to understand the distribution of the optimal response $y(x,\omega)$ induced by the random vector $c:\Omega\to \Sph_{\ny}$, we observe that its optimality\AS{ with regards to the lower-level problem} is characterized by $-c(\omega)\in N_{S(x)}(y(x,\omega))$. Thus, we consider the belief $\beta:X\to \mathscr{P}(Y)$ over the fiber map $S:X\tto Y$ given by the density
\begin{equation}\label{eq:VertexSupportedDensity}
	d\beta_x(y) := p_c(x,y) := \P[  -c(\omega) \in N_{S(x)}(y) ]. 
\end{equation}
Note that  $\P[  -c(\omega) \in N_{S(x)}(y) ] =  \P[  -c(\omega) \in \into(N_{S(x)}(y))\cap \Sph_{\ny} ]$ for any $(x,y)\in D$, and that $\into(N_{S(x)}(y))$ is nonempty if and only if $y$ is an extreme point of $S(x)$. Therefore, one can easily deduce that for each $x\in X$, the function $p_{c}(x,\cdot)$ is a discrete density with its support contained in $\ext(S(x))$ and the belief $\beta$ can be written as
\begin{equation}\label{eq:VertexSupportedBelief}
	\beta_x(O) = \sum_{y\in \ext(S(x))} p_c(x,y)\ind_O(y),\quad\forall O\in\mathcal{B}(Y).
\end{equation}
We call $\beta$ the \emph{vertex-supported belief induced by $c$}. With this construction, we can rewrite problem \eqref{eq:target-problem} as
\begin{equation}\label{eq:target-problem-Bayesian}
		\min_{x\in X}\quad \theta(x):=\,\, \langle d_1,x\rangle + \E_{\beta_x}[\langle d_2,\AS{\cdot}\rangle]\\
\end{equation}
where $\E_{\beta_x}[\langle d_2,\AS{\cdot} \rangle] = \sum_{y\in \ext(S(x))} \langle d_2,y\rangle p_c(x,y)$. Our goal in this work is to study problem \eqref{eq:target-problem} by profiting from the Bayesian formulation \eqref{eq:target-problem-Bayesian}, in the sense of \cite{SalasSvensson2020Existence}.
\begin{remark}
	Note that, by defining the centroid of $S(x)$ with respect to $\beta_x$ as
	\begin{equation}
		\mathfrak{b}(x) := \E_{\beta_x}[y] =  \sum_{y\in \ext(S(x))} p_c(x,y) y,
	\end{equation}
	we get that $\theta(x) = \langle d_1,x\rangle + \langle d_2,\mathfrak{b}(x)\rangle$. Thus, the Bayesian formulation \eqref{eq:target-problem-Bayesian} can be rewritten in terms of the centroid, which is the convex combination of the vertices of $S(x)$, proportionally to the weight of their normal cones. Of course, this is valid since the expected value $\E_{\beta_x}[\langle d_2,y\rangle]$ has a linear integrand.
\end{remark}


\subsection{The chamber complex}

We start by recalling the definition of polyhedral complex and 
chamber complex, a well-known concepts in 
computational geometry (see, e.g.,  \cite{de2010triangulations}).
	\begin{definition}[Polyhedral complex]\label{def:polycomplex} A set of polyhedra $\mathscr{P}$ in $ \R^{\nx}$ is a polyhedral complex if
		\begin{enumerate}
			\item For each $P,Q\in \mathscr{P}$, $P\cap Q$ is a face of $Q$ and of $P$.
			\item For each $P\in \mathscr{P}$ and each face $F$ of $P$, $F\in\mathscr{P}$. 
		\end{enumerate}
	\end{definition}
Given $P\in \mathscr{P}$, each vertex of $P$ (as a face of $P$) belongs to $\mathscr{P}$ and moreover, these are the minimal elements of $\mathscr{P}$. If the maximal elements in $\mathscr{P}$ have the same dimension, we say that $\mathscr{P}$ is a pure polyhedral complex. We are interested in a class of polyhedral complexes that are pure, the so-called chamber complexes, which we recall next.

\begin{definition}[Chamber complex]\label{def:Chambers} Let $D\subseteq \R^{\nx}\times\R^{\ny}$ be a polyhedron as described in \eqref{eq:HighPoint-Relax}. For each $x\in X=\pi(D)$, we define the chamber of $x$ as
	\begin{equation}\label{eq:ChamberOfAPoint}
		\sigma(x) := \bigcap\big\{ \pi(F)\, :\, F\in \mathscr{F}(D),\, x\in \pi(F) \big\}. 
	\end{equation}
	The \textbf{chamber complex} is the (finite) collection of chambers, i.e.,
	\begin{equation}\label{eq:ChamberComplex}
		\mathscr{C}(D) := \{ \sigma(x) \, :\, x\in X \}.
	\end{equation}
\end{definition}

The chamber complex is a polyhedral complex and so $\{\ri(C):C\in\mathscr{C}(D)\}$ is a partition of $X$.
The next proposition provides additional facts about the chamber of a point.

\begin{proposition}\label{prop:RequiredFaces} For any $x\in X$, one has that
	\begin{enumerate}
		\item $\sigma(x) = \bigcap\{ \pi(F)\ :\ F\in \mathscr{F}_{\leq \nx},\, x\in \pi(F) \}$.
		\item For any $C\in\mathscr{C}(D)$,  $x\in \ri(C)$ if and only if $C=\sigma(x)$.
	\end{enumerate}
\end{proposition}
\begin{proof} $1.$ Let us consider $F\in\mathscr{F}$ with $x\in \pi(F)$ such that $\adim(F) = n > \nx$. Then, there exists $J\subseteq [m]$ with $|J| = (\nx+\ny)-n < \ny$ such that $F = D\cap [g_J=0]$. Now, choose $y\in S(x)$ such that the point $(x,y)\in F$ is a vertex of $\{x\}\times S(x)$. Let $\hat{J}\subseteq [m]$ be the set of all active constraints at $(x,y)$. Then, $J\subseteq \hat{J}$ and 
	$\{(x,y)\} = [g_{\hat{J}}(x,\cdot) = 0]$. This implies that 
	\[
	\{ \nabla g_j\ :\ j\in\hat{J}\}\cup \{ (e_k,0) \in \R^{\nx}\times \R^{\ny}\ :\ k\in [\nx] \},
	\]
	has $\nx+\ny$ linearly independent vectors. In particular, $\{ \nabla g_j\ :\ j\in\hat{J}\}$ has at least $\ny$ linearly independent vectors, which yields that $\hat{F} = D\cap[g_{\hat{J}}= 0]$ satisfies that $\adim(\hat{F})\leq \nx$. Then, since $J\subseteq \hat{J}$, we have that $\hat{F} \subseteq F$, which yields that $\pi(\hat{F})\subseteq\pi(F)$. The arbitrariness of $F$ allows us to write
	\begin{align*}
		\sigma(x) = \bigcap\{ \pi(F)\ :\ F\in \mathscr{F},\, x\in \pi(F)\}
		\supseteq \bigcap\{ \pi(F)\ :\ F\in \mathscr{F}_{\leq\nx},\, x\in \pi(F) \},
	\end{align*}
	finishing the proof of the first assertion, since the reverse inclusion is direct. 
	
	2. Assume first that $C=\sigma(x)$. Since $\{\ri(C')\ :\ C'\in \mathscr{C}(D)\}$ is a partition of $X$, we know there exists $z\in X$ such that $x\in\ri(\sigma(z))$. Assume by contradiction that $x\notin\ri(C)$ which implies that $z\neq x$. Then clearly $\sigma(x)\subseteq \sigma(z)$ and so $\sigma(x)$ is a proper face of $\sigma(z)$, so that $\sigma(x)\,\cap\,\ri(\sigma(z))=\emptyset$. But this implies that $x\notin \sigma(x)$, which is evidently a contradiction.
	
	Conversely assume that $x\in\ri(C)$. Since $x\in\ri(\sigma(x))$ we get that $\ri(\sigma(x))\cap \ri(C)\neq \emptyset$, and recalling again that $\{\ri(C')\ :\ C'\in \mathscr{C}(D)\}$ is a partition of $X$ we conclude that $C=\sigma(x)$. 
\qed	
\end{proof}

In what follows, it will be useful to distinguish the maximal and the minimal chambers, with respect to the inclusion order. Minimal chambers are the zero-dimensional chambers, while maximal chambers coincide with those with nonempty interior, since $D$ has nonempty interior. The following definition introduces such distinction.

\begin{definition}\label{def:ZeroFullChambers} Let $D\subseteq \R^{\nx}\times\R^{\ny}$ be a polyhedron as described in \eqref{eq:HighPoint-Relax}. We define the families
	\begin{align}
		\mathscr{K}(D) &:= \left\{ K\in \mathscr{C}(D)\ :\ \into(K)\neq\emptyset \right\},\label{eq:FullDimensionalChambers}\\
		\mathscr{V}(D) &:= \left\{ v\in X\ :\ \{v\}\in \mathscr{C}(D) \right\}.\label{eq:ZeroDimensionalChambers}
	\end{align}
	We call $\mathscr{K}(D)$ the full-dimensional chambers of $\mathscr{C}(D)$, and $\mathscr{V}(D)$, the chamber vertices (or zero-dimensional chambers) of $\mathscr{C}(D)$.
\end{definition}

It is worth mentioning that $\mathscr{K}(D)$ is a covering of $X$, 
i.e., $X= \bigcup_{K\in \mathscr{K}(D)} K.$ This follows from the fact that each chamber is contained in a maximal (full-dimensional) chamber. 

\AS{
\begin{remark}\label{rem:ex1}
    Consider $D \subseteq \R^{2}\times\R$ the polytope defined as tuples $(x_1,x_2,y)$ such that 
\(
x_1-y\leq 0,\; -x_1-y\leq 0,\; \DS{x_2+y\leq 1},\text{ and }-x_2+y\leq 1
\); \GM{these are the same} constraints of the problem in Example \ref{ex:banana}.

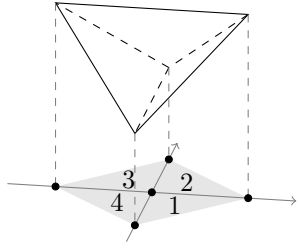
\begin{figure}[ht]
\centering
\tdplotsetmaincoords{110}{-10}
\begin{tikzpicture}[scale=1.3,
    tdplot_main_coords,
    axis/.style={->,gray},
    vector/.style={-stealth,very thick}, 
    vector guide/.style={dashed,red,thick},
    grid/.style={very thin,gray}
    ]
     \def\zshift{-1};

    \coordinate (p1) at (1,0,\zshift);
    \coordinate (p2) at (0,1,\zshift);
    \coordinate (p3) at (-1,0,\zshift);
    \coordinate (p4) at (0,-1,\zshift);
    \fill[color=black!10!white] (p1) -- (p2) -- (p3) -- (p4) -- cycle;
    \coordinate (v1) at (1,0,1);
    \coordinate (v2) at (0,1,0);
    \coordinate (v3) at (-1,0,1);
    \coordinate (v4) at (0,-1,0);
    \draw (v1) -- (v2) -- (v3) --cycle;
    \draw[dashed] (v1) -- (v4) -- (v2);
    \draw[dashed] (v4) -- (v3);

    \coordinate (O) at (0,0,\zshift);
    
    \draw[dashed,gray] (v1) -- (p1);
    \draw[dashed,gray] (v2) -- (p2);
    \draw[dashed,gray] (v3) -- (p3);
    \draw[dashed,gray] (v4) -- (p4);

    \draw[axis] (-1.5,0,\zshift) -- (1.5,0,\zshift);
    \draw[axis] (0,1.5,\zshift) -- (0,-1.5,\zshift);

    \filldraw (0,0,\zshift) circle (1pt);
    \filldraw (p1) circle (1pt);
    \filldraw (p2) circle (1pt);
    \filldraw (p3) circle (1pt);
    \filldraw (p4) circle (1pt);

    \node at (0.3,0.35,\zshift) {1};
    \node at (0.3,-0.35,\zshift) {2};
    \node at (-0.3,-0.35,\zshift) {3};
    \node at (-0.3,0.35,\zshift) {4};
    
\end{tikzpicture}
\caption{The polytope $D$ of Remark \ref{rem:ex1} and its chamber complex.}
\label{fig:banana}
\end{figure}

Clearly the vertices of $D$ are $v_1=(1,0,0)$, $v_2=(0,1,1)$, $v_3=(-1,0,0)$ and $v_4=(0,-1,1)$, and their projections $(1,0)$, $(-1,0)$, $(0,1)$ and $(0,-1)$ are chamber vertices. Additionally, $(0,0)$ is a vertex of the chamber complex $\mathscr{C}(D)$ since it is the intersection of the projections of the edges between $v_1$ and $v_3$ and between $v_2$ and $v_4$. 

In Figure \ref{fig:banana}, the 5 chamber vertices are depicted using dots while its 4 full-dimensional \DS{chambers} are enumerated 1 to 4. \hfill$\Diamond$
\end{remark}}

We close this preliminary section with the following direct proposition that relates minimal and maximal chambers.

\begin{proposition}\label{prop:MinimalChambersAreExtremePoints}
	Given $D$ as in \eqref{eq:HighPoint-Relax} with nonempty interior then
	$\mathscr{V}(D)$ coincides with $\bigcup_{K\in \mathscr{K}(D)} \ext(K)$.
\end{proposition}

\begin{proof}
	Since $\mathscr{C}(D)$ is a polyhedral complex (see, e.g., \cite{de2010triangulations}), a vertex of a given chamber is itself a chamber, and hence an element of $\mathscr{V}(D)$. This proves the inclusion $\mathscr{V}(D) \supseteq \bigcup_{K\in \mathscr{K}(D)} \ext(K)$.
		
		For the reverse inclusion,
		take $v\in \mathscr{V}(D)$. Since $X =\bigcup_{K\in \mathscr{K}(D)} K$, then $\{v\}$ is a chamber that is included in some maximal chamber $K\in\mathscr{K}(D)$. But then $\{v\}$ is a vertex of $K$ again from the property of polyhedral complexes. This finishes the proof. \qed
\end{proof}

\section{Geometrical structure of vertex-supported belief \label{sec:Chambers}}

From the properties of the chamber complex, it is easy to deduce that continuous piecewise linear functions, where the pieces are the chambers, attain their extreme values on $\mathscr{V}(D)$. Definition \ref{def:PiecewiseLinearFunction} and Corollary \ref{cor:PiecewiseLinearHasSolution} formalize this intuition. 
\begin{definition}\label{def:PiecewiseLinearFunction}
A function $f:X\to \R$ is said to be \emph{piecewise linear over the chamber complex $\mathscr{C}(D)$} if there exists a sequence of pairs $\{ (d_{C},a_{C})\ :\ C \in \mathscr{C}(D)  \}\subseteq \R^{\nx}\times \R$ such that
\begin{equation}\label{eq:PiecewiseLinearRep}
		f(x) = \sum_{C\in \mathscr{C}(D)} (\langle d_{C},x\rangle +a_{C}) \ind_{\ri(C)}(x),\qquad\forall x\in X.
\end{equation}
\end{definition}

\begin{corollary}\label{cor:PiecewiseLinearHasSolution} If a function $f:X\to \R$ is continuous and piecewise linear over the chamber complex  $\mathscr{C}(D)$, then it has at least one minimizer in $\mathscr{V}(D)$.	
\end{corollary}
\begin{proof}
Since $f$ is continuous and $X$ is nonempty and compact, Weierstrass theorem entails that $f$ attains its minimum at some point $x^*\in X$. Then, there exist a chamber $C\in \mathscr{C}(D)$ such that $x^*\in C$. Since $f$ is piecewise linear over the chamber complex and continuous, we have that $\traza{f}{C}$ is affine. Thus, there exists $v^*\in \ext(C)$ such that $f(v^*) = \traza{f}{C}(v^*) \leq  \traza{f}{C}(x^*) = f(x^*)$. Clearly $\ext(C)\subseteq \ext(K)$ for some $K\in \mathscr{K}(D)$ and hence Proposition \ref{prop:MinimalChambersAreExtremePoints} yields that $v^*\in \mathscr{V}(D)$, which finishes the proof.   \qed
\end{proof}

\subsection{Piecewise affinity of the fiber map}

We show in this section that the fiber map $S$ defined in equation \eqref{eq:sliceMap} is piecewise affine, in the sense of Definition~\ref{def:AffineMap} below. Some properties in this subsection seem to be well-known but are hard to find in the literature in the present form. Here, we provide a presentation (and proofs) that are exclusively based on mathematical programming tools. We also present one of our main contributions, which is Theorem~\ref{thm:piecewise-Bayesian}.

\begin{definition}\label{def:AffineMap}
	A set-valued map $M:X\subseteq\R^{\nx}\tto\R^\ny$ is \emph{affine} on a convex set $K\subseteq \dom(M)$ if
	\begin{equation*}
		M(\eta x_1+(1-\eta)x_0)=\eta M(x_1)+(1-\eta)M(x_0),
    	\label{eq:affinesetvaluedmap}
	\end{equation*}
	for all $x_0,x_1\in K$ and $\eta\in(0,1)$, where the addition of sets is understood in the sense of Minkowski. We say that $M$ is \emph{piecewise affine} if the domain of $M$ can be written as a finite union of convex subsets where, in each one of them, the set-valued map is affine. 
\end{definition}

Note that if a set-valued map $M$ is affine over $K$ and there exists $\bar{x}\in \ri(K)$ such that $|M(\bar{x})|=1$, then $M$ must be single-valued over $K$. Indeed, consider $x\in K$ and  consider arbitrary points $y',y''\in M(x)$. Since $\bar{x}\in\ri(K)$ we can find $\tilde{x}\in K$ and $\eta\in(0,1)$ such that $\bar{x}=\eta \tilde{x}+(1-\eta)x.$ We can take $\tilde{y}\in M(\tilde{x})$. But then both $\bar{y}'=\eta \tilde{y}+(1-\eta)y'$ and $\bar{y}''=\eta \tilde{y}+(1-\eta)y''$ belong to $M(\bar{x})$ which implies that $\bar{y}'=\bar{y}''$ and this yields that $y'=y''$. So $M(x)$ contains a single point, and from the arbitrariness of $x$, we conclude that $M$ is single-valued in $K$. With this observation, we can establish the following lemma.

\begin{lemma}\label{lem:affineprops}
Let $M:X\subseteq \R^{\nx}\tto \R^{\ny}$ be affine on a convex set $K\subseteq \dom(M)$, and consider a linear functional $\psi:\R^\ny\to\R$ such that for some $x\in K$, $\sup\psi(M(x))$ is finite and attained. Then:
\begin{enumerate}
    \item The function $\varphi(x):=\sup\psi(M(x))$ is affine \AS{on $K$}.
    \item The set-valued map $M_\psi(x):=\argmax_{y\in M(x)} \psi(y)$ is affine on $K$. Additionally, if there exists a point $\bar{x}\in\ri(K)$, with $|M_\psi(\bar{x})|=1$, then $M_{\psi}$ is single-valued on $K$.
\end{enumerate}
\end{lemma}

\begin{proof}
Fix $x_0,x_1\in K$ and $\eta \in (0,1)$, and let $x_{\eta} = \eta x_1 + (1-\eta)x_0$. The first part of the Lemma follows from 
    \begin{align*}
    \sup\psi(M(x_\eta))&=\sup\psi(\eta M(x_1)+(1-\eta)M(x_0))\\
    &=\sup \{\eta \psi(M(x_1))+(1-\eta)\psi(M(x_0))\}\\
    &=\sup \{\eta \psi(M(x_1))\}+\sup \{(1-\eta)\psi(M(x_0))\}\\
    &=\eta\varphi(x_1)+(1-\eta)\varphi(x_0).
    \end{align*}

Let us now prove the second part. Indeed, if $y\in M_\psi(x_\eta)$, we have that $y\in M(x_\eta)$ and since $M$ is affine, then $y=\eta y_1+(1-\eta)y_0$ for some $y_0\in M(x_0)$ and $y_1\in M(x_1)$. Since $\varphi$ is affine we have that
    \begin{align*}
        \varphi(x_\eta)&=\eta\varphi(x_1)+(1-\eta)\varphi(x_0)\\
        &\geq \eta\psi(y_1)+(1-\eta)\psi(y_0)=\psi(y)=\varphi(x_\eta).
    \end{align*}
    But then $\eta\varphi(x_1)+(1-\eta)\varphi(x_0)= \eta\psi(y_1)+(1-\eta)\psi(y_0)$ and this can only occur if $\psi(y_1)=\varphi(x_1)$ and $\psi(y_0)=\varphi(x_0)$. Thus we have $y_1\in M_\psi(x_1)$ and $y_0\in M_{\psi}(x_0)$. We have hence proved the inclusion 
    \[
    M_\psi(x_\eta)\subseteq \eta M(x_1)+(1-\eta)M(x_0).
    \]
    The reverse inclusion corresponds to the convexity of the graph of $M_\psi$, which follows directly from the fact that the  $\gph M_\psi$ can be written as the intersection of
    $\gph M$  and $\{(x,y):\psi(y)-\varphi(x)\geq 0\}$.
    The fact that $M_{\psi}$ is single-valued if there exists a point $\bar{x}\in\ri(K)$ such that $|M_\psi(\bar{x})|=1$ follows from the discussion preceding the Lemma statement.  \qed
\end{proof}

\begin{lemma}\label{lem:PiecewiseAffine-S_PI}
	The fiber map $S$ defined in equation \eqref{eq:sliceMap} is piecewise affine over the chamber complex $\mathscr{C}(D)$.	More precisely, for any chamber $K\subseteq \mathscr{C}(D)$,  $S$ is affine on $K$. 
\end{lemma}
\begin{proof}
\
Let $K\in \mathscr{C}(D)$ and let $v_1,\ldots,v_k$ be the vertices of $K$. If $x_0,x_1\in K$ then $$x_0=\sum_{i=1}^k \lambda_{0,i}v_i\quad x_1=\sum_{i=1}^k \lambda_{1,i}v_i$$
for some $\lambda_{0,i},\lambda_{1,i}\geq 0$, $i\in[k]$, such that $\sum_{i=1}^k\lambda_{0,i}=\sum_{i=1}^k\lambda_{1,i}=1$. Consider now a point $x_{\eta}:=\eta x_1+(1-\eta)x_0$ for some $\eta\in (0,1)$. We clearly have that $x_{\eta}=\sum_{i=1}^k\lambda_{\eta,i} v_i$
with $\lambda_{\eta,i}:=\eta \lambda_{0,1}+(1-\eta)\lambda_{0,i}$, $i\in[k]$.

Applying the representation of \cite[Proposition 2.4]{rambau1996projections}, we know that for each convex combination $x = \sum_{i=1}^k \lambda_iv_i \in K$, we have that $S(x) = \sum_{i=1}^k \lambda_i S(v_i)$. Thus, we can write
\begin{align*}
S(x_{\eta})=\sum_{i=1}^k\lambda_{\eta,i}S(v_i)&=\eta\sum_{i=1}^k\lambda_{1,i}S(v_i)+(1-\eta)\sum_{i=1}^k\lambda_{0,i}S(v_i)\\
&=\eta S(x_1)+(1-\eta) S(x_0).
\end{align*} 
This concludes the proof.\qed
\end{proof}

Recall from \cite{SalasSvensson2020Existence} that a belief $\beta:X\to\mathscr{P}(Y)$ is said to be \textit{weak continuous} if for every sequence $(x_k)\subseteq X$ converging to $x\in X$ the measure $\beta_{x_k}$ weak converges to $\beta_x$ (in the sense of \cite[Chapter 13]{klenke2013probability}), that is,
\begin{equation}\label{eq:WeakConvergence}
\int_Y f(y)d\beta_{x_k}(y) \xrightarrow{k\to\infty}\int_Y f(y) d\beta_{x}(y),\quad \forall f:Y\to \R \text{ continuous.}
\end{equation}

\begin{theorem} \label{thm:piecewise-Bayesian}
 Consider a random cost $c:\Omega\to \Sph_{\ny}$ with nonatomic distribution and the vertex-supported belief $\beta:X\to \mathscr{P}(Y)$ over $S$ induced by $c$ as defined in \eqref{eq:VertexSupportedBelief}. Then,
\begin{enumerate}
    \item $\beta$ is weak continuous, and thus for any lower semicontinuous function $f:~X\times Y\to \R$, the problem $\displaystyle\min_{x\in X}\, \E_{\beta_x}[f(x,\cdot)]$ has a solution.
    \item The function $\theta:X\to \R$ given by $\theta(x) := \langle d_1,x\rangle + \E_{\beta_x}[ \langle d_2,\cdot\rangle ]$ is continuous and piecewise linear over $\mathscr{C}(D)$.
\end{enumerate}
In particular, problem \eqref{eq:target-problem} has at least one solution in $\mathscr{V}(D)$.
\end{theorem}
\begin{proof}
    Since $\mathscr{K}(D)$ is a covering of $X$, then it is enough to prove the assertions within a given chamber $K\in\mathscr{K}(D)$. Let $\bar{x}\in\into K$ and consider the vertices  $\bar{y}_1,\ldots,\bar{y}_k$ of $S(\bar{x})$. Since $S$ is affine on $K$ with compact polyhedral images, we first claim that there exist affine functions $y_1,y_2,\ldots,y_k:K\to Y$, such that for all $x\in K$ one has that $\{y_i(x)\}_{i\in[k]}= \ext(S(x))$ or, equivalently, that
    \[
   \{y_i(x)\}_{i\in[k]}\subseteq \ext(S(x))\quad\text{ and }\quad S(x)=\mathrm{conv}\{y_i(x)\ :\ i\in[k]\}.
    \]
    Indeed, 
    for each $i\in[k]$ there exists a linear functional $\psi_i$ that exposes $\bar{y}_i$, i.e., such that $\{\bar{y}_i\}=\{y\in S(\bar{x}):\psi_i(y)\geq \psi_i(\bar{y}_i)\}$. 
    Then from Lemma \ref{lem:affineprops} the maximum of $\psi_i$ over $S(x)$ defines the affine function 
    \[
    x\mapsto y_i(x)=\argmax\{\psi_i(y):y\in S(x)\},
    \]
    for each $i\in[k]$. This proves that $\{y_1(x),y_2(x),\ldots,y_k(x)\}\subseteq \ext(S(x))$ for each $x\in K$. Now, it is direct that $S(x)\supseteq \mathrm{conv}(y_1(x),y_2(x),\ldots,y_k(x))$, for each $x\in K$, and so we only need to prove the direct inclusion. Suppose that this is not the case and choose $x\in K$ and $y\in S(x)\setminus M$ with $M = \mathrm{conv}(y_1(x),y_2(x),\ldots,y_k(x))$. Since $M$ is compact, there exist a linear functional $\psi$ that strictly separates $y$ from $M$, and so $\sup \psi(S(x))>\sup \psi(M)$. Now, since $\bar{x}\in \into (K)$, there exists $x'\in K$ and $\eta\in (0,1)$ such that $\bar{x} = \eta x + (1-\eta)x'$. Let $M' = \mathrm{conv}(y_1(x'),y_2(x'),\ldots,y_k(x'))$. Then, by affinity of the functions $y_1,\ldots,y_k$, one has that $S(\bar{x}) = \eta M + (1-\eta)M'$. Then, by Lemmas \ref{lem:affineprops} and \ref{lem:PiecewiseAffine-S_PI} we can write
    \begin{align*}
        \sup \psi(S(\bar{x})) &=  \sup \psi(\eta M + (1-\eta)M')\\
        &= \eta\sup \psi(M) + (1-\eta)\sup \psi(M')\\
        &< \eta\sup \psi(S(x)) + (1-\eta)\sup \psi(S(x')) = \sup \psi(S(\bar{x})),
    \end{align*}
    which is a contradiction. This proves our first claim. Our second claim is that for each $i\in[k]$
    \begin{equation}\label{eq:NormalCones_Function}
       N_{S(x)}(y_i(x)) = \bigcup_{j\in[k]: y_j(x)=y_i(x)} N_{S({\bar{x}})}(\bar{y}_j),  \quad \forall x\in K. 
    \end{equation}
    Indeed, for the direct inclusion, fix $x\in K$ and choose $\psi\in \into(N_{S(x)}(y_i(x)))$, which is nonempty since $y_i(x)\in \ext(S(x))$. This means that $\psi$ attains its maximum on $S(x)$ only at $y_i(x)$. Now, $\psi$ attains its maximum over $S(\bar{x})$ at some $\bar{y}_j$, and so $\psi \in N_{S({\bar{x}})}(\bar{y}_j)$. Since $\bar{x}\in \into(K)$, there exists $x'\in K$ and $\eta\in (0,1)$ such that $\bar{x} = \eta x + (1-\eta)x'$. Then, by construction, $\bar{y}_j= \eta y_j(x) + (1-\eta)y_j(x')$, which yields that $\psi$ attains its maximum on $S(x)$ at $y_j(x)$. This implies that $y_i(x) = y_j(x)$, and so $\psi \in \bigcup_{j\in[k]: y_j(x)=y_i(x)} N_{S({\bar{x}})}(\bar{y}_j)$. The inclusion is proven by taking closure of $\into(N_{S(x)}(y_i(x)))$. Now, for the reverse inclusion, choose $j\in [k]$ such that $y_i(x) = y_j(x)$ and $\psi\in N_{S(\bar{x})}(\bar{y}_j)$. Then, $\psi$ attains its maximum on $S(\bar{x})$ at $\bar{y}_j$ and so, replicating the argument above, we have that $\psi$ attains its maximum on $S(x)$ at $y_j(x)$. Since $y_i(x) = y_j(x)$, the conclusion follows. This proves the second claim.
    
    Now, \eqref{eq:NormalCones_Function} entails that $x\mapsto N_{S(x)}(y_i(x))$ is constant in $\into (K)$ since for every $x\in \into(K)$, one has that $y_i(x)\neq y_j(x)$ if and only if $i\neq j$. Then for each $i\in[k]$ it is clear that $$p_c(x,y_i(x))=\P [-c(\omega)\in N_{S(x)}(y_i(x))]$$ is constant in $\into (K)$, equal to 
    \begin{equation}
    \label{eq:defofci}
    c_i:=p_c(\bar{x},y_i(\bar{x})),    
    \end{equation} 
    and for points $x$ in the boundary of $K$ it holds $ p_c(x,y_i(x))=\sum\{c_j:j\in[k],\,  y_i(x)=y_j(x)\}.$
    Thus, it follows that for any continuous function $h:Y\to \R$ we have
    \begin{equation}
        \label{eq:vertexsuppvalue}
        \E_{\beta_{x}}[h]=\sum_{i=1}^kh(y_i(x))c_i, \qquad \forall x\in K
    \end{equation}

    This last expression is continuous on $K$. Then, $x\mapsto \E_{\beta{x}}[h]$ must be continuous on $X$ and, since $h$ is arbitrary,  
    $\beta$ is a weak continuous belief. Then, for every lower semicontinuous function $f:X\times Y\to \R$, $\E_{\beta_x}[f(x,\cdot)]$ is lower semicontinuous as well (see \cite{SalasSvensson2020Existence}). Finally, from  \eqref{eq:vertexsuppvalue} it is clear that taking $f(x,y)=\langle d_1,x\rangle + \langle d_2,y\rangle $, the function $\theta(x)=\E_{\beta_x}[f(x,\cdot)]$ is affine over $K$.\qed
\end{proof}

\subsection{Sample average formulation}
Problem \eqref{eq:target-problem} has an intrinsic difficulty which consists in how to evaluate the objective function  $\theta(x) = \langle d_1,x\rangle + \E_{\beta_x}[ \langle d_2,\cdot\rangle ]$. To make an exact evaluation of $\theta$ at a point $x\in X$ one would require to compute all the vertices $y_1,\ldots,y_k$ of $S(x)$ and the values $c_1,\ldots,c_k$ defined as the ``sizes'' of the respective normal cones at each vertex $y_i$, as in \eqref{eq:defofci}. This is not always possible.\\

To deal with this issue, we consider the well-known sample average approximation (SAA) method for stochastic optimization (see, e.g., \cite{Shapiro2021Lecture3rd,HomemdeMelloBayraksan2014Survey}). We take an independent sample $\{\hat{c}_1,\ldots,\hat{c}_{N_0}\}$ of the random lower-level cost $c(\cdot)$ 
and try to solve the (now deterministic) problem
\begin{equation}\label{eq:target-problem-SAA}
\left\{\begin{array}{cl}
\min_x & \langle d_1,x\rangle + \frac{1}{N_0} \sum_{i=1}^{N_0} \langle d_2,y_i(x)\rangle\\
\\
s.t. & \DS{A_lx \leq b_l}\\
& \forall i \in [N_0],\, y_i(x)\text{ solves }\left\{\begin{array}{cl}
    \min_y & \langle \hat{c}_i, y \rangle \\
     s.t. & \DS{A_fx + B_fy \leq b_f}. 
\end{array}\right.
\end{array}\right.
\end{equation}

\begin{proposition}\label{prop:SingleValued} Assume that $c(\cdot)$ has a nonatomic distribution over $\Sph_{\ny}$, in the sense of \eqref{eq:NonatomicDistribution}. Then, with probability 1, we have that
\[
 \argmin_{y}\left\{ \langle \hat{c}_i, y\rangle\ :\  \DS{A_fx + B_fy \leq b_f} \right\} \text{ is a singleton}\quad \forall i\in [N_0],\,\forall x\in X.
\]
\end{proposition}

\begin{proof}
	
We begin by noting that this result is well-known for a fixed $x\in X$. Here we extend it to handle the universal quantifier over $x$.
\DS{
Consider the set
\[
\mathcal{N}:=\big\{\omega\in\Omega\, :\, \exists x\in X,\, |\argmin_{y}\left\{ \langle \hat{c}_i(\omega), y\rangle :  A_fx + B_fy \leq b_f \right\}| > 1\big\}.
\]
Let $\omega\in \mathcal{N}$, and choose $x\in X$ as in the definition of $\mathcal{N}$. The dual of  the LP $\min_{y}\left\{ \langle \hat{c}_i(\omega), y\rangle :  B_fy \leq b_f - A_fx \right\}$ is 
	\[
	\max_{\rho}\left\{ \rho^T (b_f-A_fx) \,: \,  B_f^T \rho = \hat{c}_i(\omega),\, \rho \leq 0\right\}.
	\]
Take $\rho^*$ an optimal basic solution to the dual. This solution must have at most $\ny$ non-zeros; moreover, after potentially permuting columns of $B_f^T$, we can assume $\rho^* = ((\tilde{B}_f^T)^{-1} \hat{c}_i(\omega), \mathbf{0}_{m-\ny})$ where $\tilde{B}_f$ is an invertible $\ny \times \ny$ submatrix of $B_f$ and $\mathbf{0}_{m_f-\ny}$ is the vector of $m_f-\ny$ zeros. 

Let $\tilde{b}_f - \tilde{A}_fx$ be the sub-vector of $b_f - A_fx$ corresponding to the rows of $\tilde{B}_f$. 
From linear programming theory, we know there is an optimal solution to the primal $y^*$ for which $\tilde{B}_f y^* = \tilde{b}_f - \tilde{A}_fx$. 
Since we are assuming there are at least two solutions to the primal, there must exist another optimal solution $\hat{y}$ such that one of the inequalities in $\tilde{B}_f \hat{y} \leq \tilde{b}_f - \tilde{A}_fx$ has positive slack. This implies that one of the components of $(\tilde{B}_f^T)^{-1} \hat{c}_i(\omega)$ must be zero, as the primal-dual pair $(\hat{y}, \rho^*)$ must satisfy complementary slackness. We conclude that $\rho^*$ has at most $\ny - 1$ non-zeros.

Dual feasibility thus implies that $\hat{c}_i(\omega)$ is in the span of $\ny-1$ columns of $B^{\top}_f$ and such a subspace (intersected with the unit sphere $\Sph_{\ny}$) has null measure with respect to the Hausdorff measure $\mathcal{H}^{\ny-1}$. We conclude by noting that there are only finitely many such subspaces; thus, the set of $\{\hat{c}_i(\omega)\,:\, \omega\in \mathcal{N}\}$ has null measure.} \qed
\end{proof}

Based on the above proposition, for the sample $\{\hat{c}_1,\ldots,\hat{c}_{N_0}\}$ we can, for each $i\in [N_0]$, define the mapping $x~\mapsto~y_i(x)$  where $y_i(x)$ is the unique element in $\argmin_{y}\left\{ \langle \hat{c}_i, y\rangle\ :\  A_fx + B_fy \leq b_f \right\}$. The following direct theorem shows that the objective of the sample average approximation has also the property of being (almost surely) piecewise linear over the chamber complex $\mathscr{C}(D)$. 
\begin{theorem}\label{thm:piecewise-SAA}
If $c$ has a nonatomic distribution over $\Sph_{\ny}$, then, with probability 1, the function $\theta: X\to \R$ given by $\theta(x) := \langle d_1,x\rangle +\frac{1}{N}\sum_{i=1}^N \langle d_2,y_i(x)\rangle$ is well-defined and it is piecewise linear over the chamber complex $\mathscr{C}(D)$.
\end{theorem}

We close this section by showing that problem \eqref{eq:target-problem-SAA} is a consistent approximation of the original problem \eqref{eq:target-problem-Bayesian}.
\begin{proposition} Let $x^*$ be an optimal solution of problem \eqref{eq:target-problem-SAA} for a given sample of size $N_0$, and let $\nu^*$ be the associated optimal value. Let $\mathcal{S}$ be the set of solutions of problem \eqref{eq:target-problem} and let $\bar{\nu}$ be its optimal value. Then,
\begin{enumerate}
    \item $d(x^*,\mathcal{S}):=\inf_{x\in \mathcal{S}}\|x^*-x\| \xrightarrow{N_0\to\infty} 0$, w.p.1.
    \item $\nu^*\xrightarrow{N_0\to\infty} \bar{\nu}$, w.p.1.
\end{enumerate}
\end{proposition}
\begin{proof}
Define the function $f_{N_0}(x):=\frac{1}{N_0}\sum_{i=1}^{N_0} \langle d_2,y_i(x)\rangle$. First, note that, with probability 1, $f_{N_0}(x)$
converges to the function $f(x):=
\E_{\beta_x}[\langle d_2,\cdot\rangle]$ pointwisely as $N_0\to \infty$. Since $\mathscr{V}(D)$ is finite, we have that, with probability 1, $f_{N_0}$ converges to $f$ uniformly on $\mathscr{V}(D)$. That is, for almost every infinite sample $\{ \hat{c}_i(\omega) \}_{i\in\N}$, one has that $\sup_{x\in \mathscr{V}(D)}| f_{N_0}(x) - f(x)|\to 0$ as $N_0\to \infty$.

Consider a full-dimensional chamber $K\in\mathscr{K}(D)$. Then affinity of $f_{N_0}$ and $f$ on $K$ yields that
\[
\sup_{x\in K}| f_{N_0}(x) - f(x)| = \sup_{x\in \ext(K)}| f_{N_0}(x) - f(x)|\leq \sup_{x\in \mathscr{V}(D)}| f_{N_0}(x) - f(x)|,
\]
where the last inequality follows from the inclusion $\ext(K)\subseteq \mathscr{V}(D)$ (see Proposition \ref{prop:MinimalChambersAreExtremePoints}). Thus, with probability 1, the convergence is uniform on $K$. Since there are finitely many full-dimensional chambers and they cover $X$, we conclude that the convergence is, with probability 1, uniform in $X$ as well. Thus from  \cite[Theorem 5.3]{Shapiro2021Lecture3rd} we conclude the desired properties.\qed
\end{proof}


\section{Some illustrative examples\label{sec:Examples}}

Here we provide some academic examples illustrating key ideas and difficulties of the chamber complex. We also revise in detail the bilevel continuous knapsack problem with uncertain costs, which was recently studied in \cite{BuchheimHenkeIrmai2022Knapsack}.

\subsection{Academic examples}

\begin{example}[$\mathscr{V}(D)$ can be arbitrarily large \AS{for fixed dimensions ($2\times 1$)}]
\label{ex:bananawithcuts}
Let $n\in\N$ and fix $\varepsilon\in (0,1)$. For each $i=0,\ldots,n$ define $\alpha_i=\frac{\varepsilon i}{2n^2}(2n-i+1)$ and $\beta_i=1-\frac{i}{n}$. Now consider the following polytope $D_n$ of all points $(x,y)\in \R^2\times\R$ satisfying that $\alpha_i+\beta_i|x_1|\leq    y\leq 1-\alpha_i-\beta_i|x_2|$, for all $i\in\{0,\ldots,n\}$.
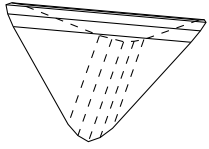
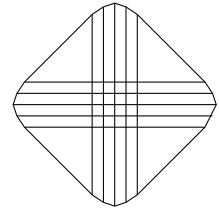
\begin{figure}[t]
\centering
\begin{subfigure}[b]{0.45\textwidth}
\centering
\tdplotsetmaincoords{120}{-10}
\begin{tikzpicture}[scale=1.5,
tdplot_main_coords,
axis/.style={->,gray},
vector/.style={-stealth,very thick}, 
vector guide/.style={dashed,red,thick},
grid/.style={very thin,gray}
]
\def\eps{0.2}; 

\coordinate (v11) at ({1-\eps},\eps,{1-\eps});
\coordinate (v31) at ({-(1-\eps)},\eps,{1-\eps});
\draw (v11) -- (v31);

\coordinate (v12) at ({1-\eps*2/3},{\eps/2},{1-\eps*2/3});
\coordinate (v32) at ({-(1-\eps*2/3)},{\eps/2},{1-\eps*2/3});
\draw (v12) -- (v32);

\coordinate (v13) at ({1-\eps/2},{0},{1-\eps/2});
\coordinate (v33) at ({-(1-\eps/2)},{0},{1-\eps/2});
\draw (v13) -- (v33);

\coordinate (v14) at ({1-\eps*2/3},{-\eps/2},{1-\eps*2/3});
\coordinate (v34) at ({-(1-\eps*2/3)},{-\eps/2},{1-\eps*2/3});
\draw (v14) -- (v34);

\coordinate (v15) at ({1-\eps},{-\eps},{1-\eps});
\coordinate (v35) at ({-(1-\eps)},{-\eps},{1-\eps});
\draw (v15) -- (v35);

\draw (v11) -- (v12) -- (v13) -- (v14) -- (v15);
\draw (v31) -- (v32) -- (v33) -- (v34) -- (v35);

\coordinate (v41) at (\eps,{-(1-\eps)},\eps);
\coordinate (v21) at (\eps,{1-\eps},\eps);
\draw[dashed] (v41) -- (v21);

\coordinate (v42) at ({\eps/2},{-(1-2/3*\eps)},{\eps*2/3});
\coordinate (v22) at ({\eps/2},{1-2/3*\eps},{\eps*2/3});
\draw[dashed] (v42) -- (v22);

\coordinate (v43) at (0,{-(1-\eps/2)},{\eps/2});
\coordinate (v23) at (0,{1-\eps/2},{\eps/2});
\draw[dashed] (v43) -- (v23);

\coordinate (v44) at ({-\eps/2},{-(1-2/3*\eps)},{\eps*2/3});
\coordinate (v24) at ({-\eps/2},{1-2/3*\eps},{\eps*2/3});
\draw[dashed] (v44) -- (v24);

\coordinate (v45) at (-\eps,{-(1-\eps)},\eps);
\coordinate (v25) at (-\eps,{1-\eps},\eps);
\draw[dashed] (v45) -- (v25);

\draw (v11) -- (v21) -- (v22) --(v23) --(v24) --(v25) -- (v31);
\draw[dashed] (v15) -- (v41) -- (v42) -- (v43) -- (v43) -- (v45) -- (v35);

\coordinate (O) at (0,0,-0.7);
\node at (O){};

\end{tikzpicture}
\caption{The polytope $D_2$ from Example \ref{ex:bananawithcuts}}
\end{subfigure}
\hfill
\begin{subfigure}[b]{0.45\textwidth}
\centering
\tdplotsetmaincoords{0}{-90}
\begin{tikzpicture}[scale=1.5,
tdplot_main_coords,
axis/.style={->,gray},
vector/.style={-stealth,very thick}, 
vector guide/.style={dashed,red,thick},
grid/.style={very thin,gray}
]
\def\eps{0.2}; 

\coordinate (v11) at ({1-\eps},\eps,{1-\eps});
\coordinate (v31) at ({-(1-\eps)},\eps,{1-\eps});
\draw (v11) -- (v31);

\coordinate (v12) at ({1-\eps*2/3},{\eps/2},{1-\eps*2/3});
\coordinate (v32) at ({-(1-\eps*2/3)},{\eps/2},{1-\eps*2/3});
\draw (v12) -- (v32);

\coordinate (v13) at ({1-\eps/2},{0},{1-\eps/2});
\coordinate (v33) at ({-(1-\eps/2)},{0},{1-\eps/2});
\draw (v13) -- (v33);

\coordinate (v14) at ({1-\eps*2/3},{-\eps/2},{1-\eps*2/3});
\coordinate (v34) at ({-(1-\eps*2/3)},{-\eps/2},{1-\eps*2/3});
\draw (v14) -- (v34);

\coordinate (v15) at ({1-\eps},{-\eps},{1-\eps});
\coordinate (v35) at ({-(1-\eps)},{-\eps},{1-\eps});
\draw (v15) -- (v35);

\draw (v11) -- (v12) -- (v13) -- (v14) -- (v15);
\draw (v31) -- (v32) -- (v33) -- (v34) -- (v35);

\coordinate (v41) at (\eps,{-(1-\eps)},\eps);
\coordinate (v21) at (\eps,{1-\eps},\eps);
\draw (v41) -- (v21);

\coordinate (v42) at ({\eps/2},{-(1-2/3*\eps)},{\eps*2/3});
\coordinate (v22) at ({\eps/2},{1-2/3*\eps},{\eps*2/3});
\draw (v42) -- (v22);

\coordinate (v43) at (0,{-(1-\eps/2)},{\eps/2});
\coordinate (v23) at (0,{1-\eps/2},{\eps/2});
\draw (v43) -- (v23);

\coordinate (v44) at ({-\eps/2},{-(1-2/3*\eps)},{\eps*2/3});
\coordinate (v24) at ({-\eps/2},{1-2/3*\eps},{\eps*2/3});
\draw (v44) -- (v24);

\coordinate (v45) at (-\eps,{-(1-\eps)},\eps);
\coordinate (v25) at (-\eps,{1-\eps},\eps);
\draw (v45) -- (v25);

\draw (v11) -- (v21) -- (v22) --(v23) --(v24) --(v25) -- (v31);
\draw (v15) -- (v41) -- (v42) -- (v43) -- (v43) -- (v45) -- (v35);


\end{tikzpicture}
  \caption{The chamber complex of $D_2$. 
  }
  \end{subfigure}
  \caption{Polytope of Example \ref{ex:bananawithcuts} whose chamber complex has several vertices that are not projections of vertices.}
  \label{fig:Dwithcuts}
  
\end{figure}

Note that since $\alpha_0=0$ and $\beta_0=1$, the polytope $D_n$ is a subset of $D$ \AS{in Figure \ref{fig:banana}}, but has suffered some cuts. The chamber complex of $D_n$ has $(2n+1)^2$ vertices that are not projections of vertices of the polytope itself. We depict $D_2$ and its chamber complex in Figure \ref{fig:Dwithcuts}.\hfill$\Diamond$
\end{example}

\begin{example}[$\mathscr{V}(D)$ can be exponentially larger than $\mathscr{K}(D)$]
\label{ex:expchambervertices}
Consider a polytope \AS{$\tilde{D}_n$} in $\R^{n}\times\R$, where $\nx = n$ and $\ny = 1$, defined as the convex hull of the bottom hypercube $B = \{x\in\R^n :\| x \|_{\infty} \leq 1\}\times \{ 0\}$ and  the top hypercube $T = \{x\in\R^n : \| x \|_{\infty} \leq 1/2\}\times \{1\}$. The representation of \AS{$\tilde{D}_n$} by linear inequalities can be written as
	\[
	\AS{\tilde{D}_n} = \left\{   (x,y)\ :\ \begin{array}{l}
		0\leq y \leq 1\\
		\frac{1}{2}y - 1\leq x_i \leq 1 - \frac{1}{2}y\quad\forall i\in[n].
	\end{array}
	\right\}
	\]
	
	We call $\AS{\tilde{D}_n}$ the $(n+1)$-truncated Pyramid.  Figure 	\ref{fig:Pyramid2} illustrates the truncated pyramid $\AS{\tilde{D}_2}$ and Figure \ref{fig:ChamberPyramid2} provides its chamber complex in $\R^2$. While it is no longer possible to draw $\AS{\tilde{D}_3}$ as a 4-dimensional object, Figure \ref{fig:ChamberPyramid3} shows the chamber complex in this case.  
	\vspace{0.4cm}	
	
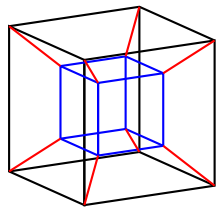
	\begin{minipage}{0.45\textwidth}
		\begin{center}
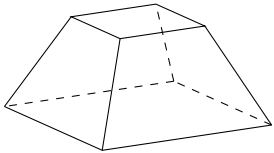

			\tdplotsetmaincoords{75}{150}
			\begin{tikzpicture}[scale=1.3,
				tdplot_main_coords,
				axis/.style={->,thick},
				vector/.style={-stealth,very thick}, 
				vector guide/.style={dashed,red,thick},
				grid/.style={very thin,gray}
				]
				
				\coordinate (O) at (0,0,0);
				
				\draw (1,-1,0)--(1,1,0)--(-1,1,0);
				\draw[dashed] (-1,-1,0)--(1,-1,0);
				\draw[dashed] (-1,1,0)--(-1,-1,0);
				
				\draw (-1/2,-1/2,1)--(1/2,-1/2,1)--(1/2,1/2,1)--(-1/2,1/2,1)--cycle;
				\draw[dashed] (-1,-1,0)--(-1/2,-1/2,1);
				\draw (1,1,0)--(1/2,1/2,1);
				\draw (1,-1,0)--(1/2,-1/2,1);
				\draw (-1,1,0)--(-1/2,1/2,1);
			\end{tikzpicture}
			\captionof{figure}{Truncated Pyramid $\AS{\tilde{D}_2}$\label{fig:Pyramid2}}
		\end{center}
	\end{minipage}
	\begin{minipage}{0.50\textwidth}
		\begin{center}
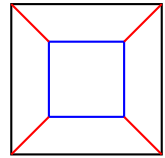

			\begin{tikzpicture}[scale=1]
				\draw[thick] (-1,-1)--(1,-1)--(1,1)--(-1,1)--cycle;
				\draw[blue,thick] (-1/2,-1/2)--(1/2,-1/2)--(1/2,1/2)--(-1/2,1/2)--cycle;
				\draw[red,thick] (-1,-1)--(-1/2,-1/2);
				\draw[red,thick] (1,-1)--(1/2,-1/2);
				\draw[red,thick] (-1,1)--(-1/2,1/2);
				\draw[red,thick] (1,1)--(1/2,1/2);
			\end{tikzpicture}
			\captionof{figure}{Chamber complex of $\AS{\tilde{D}_2}$\label{fig:ChamberPyramid2}}
		\end{center}
	\end{minipage}
	
	\begin{center}
		\tdplotsetmaincoords{75}{150}
		\begin{tikzpicture}[scale=1,
			tdplot_main_coords,
			axis/.style={->,thick},
			vector/.style={-stealth,very thick}, 
			vector guide/.style={dashed,red,thick},
			grid/.style={very thin,gray}
			]
			
			\coordinate (O) at (0,0,0);
			
			\draw[thick] (-1,1,-1)--(1,1,-1) -- (1,1,1)--(-1,1,1)--cycle;
			\draw[thick] (1,-1,-1)--(1,1,-1) -- (1,1,1)--(1,-1,1)--cycle;
			\draw[thick] (1,-1,1)--(-1,-1,1) -- (-1,1,1);
			\draw[thick] (1,-1,-1)--(-1,-1,-1) -- (-1,1,-1);
			\draw[thick] (-1,-1,-1)--(-1,-1,1);
			
			\draw[blue,thick] (-1/2,1/2,-1/2)--(1/2,1/2,-1/2) -- (1/2,1/2,1/2)--(-1/2,1/2,1/2)--cycle;
			\draw[blue,thick] (1/2,-1/2,-1/2)--(1/2,1/2,-1/2) -- (1/2,1/2,1/2)--(1/2,-1/2,1/2)--cycle;
			\draw[blue,thick] (1/2,-1/2,1/2)--(-1/2,-1/2,1/2) -- (-1/2,1/2,1/2);
			\draw[blue,thick] (1/2,-1/2,-1/2)--(-1/2,-1/2,-1/2) -- (-1/2,1/2,-1/2);
			\draw[blue,thick] (-1/2,-1/2,-1/2)--(-1/2,-1/2,1/2);
			
			\draw[red,thick] (-1,-1,-1)--(-1/2,-1/2, -1/2);
			\draw[red,thick] (1,-1,-1)--(1/2,-1/2, -1/2);
			\draw[red,thick] (-1,1,-1)--(-1/2,1/2, -1/2);
			\draw[red,thick] (1,1,-1)--(1/2,1/2, -1/2);
			\draw[red,thick] (-1,-1,1)--(-1/2,-1/2, 1/2);
			\draw[red,thick] (1,-1,1)--(1/2,-1/2, 1/2);
			\draw[red,thick] (-1,1,1)--(-1/2,1/2, 1/2);
			\draw[red,thick] (1,1,1)--(1/2,1/2, 1/2);
		\end{tikzpicture}
		\captionof{figure}{Chamber complex of $\AS{\tilde{D}_3}$\label{fig:ChamberPyramid3}}
	\end{center}
	The $(n+1)$-truncated pyramid $\AS{\tilde{D}_n}$ is an example where the number of vertices of the chamber complex $|\mathscr{V}(\AS{\tilde{D}_n})| = 2^{n+1}$ is exponentially larger than the number of full-dimensional chambers $|\mathscr{K}(\AS{\tilde{D}_n})| = 2n+1$.\hfill$\Diamond$
\end{example}

\subsection{The bilevel continuous knapsack problem}
\label{sec:continuousknapsack}

Let $a = (a_1,\ldots,a_n)\in \R_+^n$ be a weight vector and let $A := \sum_{i=1}^n a_i$. The bilevel continuous knapsack problem associated to the weight vector $a$ is given by

\begin{equation}\label{eq:knapsack}
	BK(a):=\left\{\begin{array}{cl}
		\displaystyle\max_x & -\delta x + \E[d^\top y(x,\omega)]\\
		\\
		s.t. & x\in [L,U] \\
		& y(x,\omega)\text{ solves }\left\{\begin{array}{cl}
			\displaystyle\min_{y} & \langle c(\omega), y\rangle  \\
			s.t. & \begin{array}{l}
				a^\top y \leq x,\\
				y \in [0,1]^{\ny}
			\end{array}
		\end{array}\right.\quad\text{a.s. }\omega\in\Omega,
	\end{array}\right.
\end{equation}
where $\delta\in \R_+$ is the cost of the knapsack capacity (decided by the leader), and $d\in \R^{n}$ is the valorization of the leader over the objects $1$ to $n$. Once the knapsack size $x$ is decided, the follower must choose how much of each (divisible) object to carry with, respecting the capacity constraint $a^{\top}y\leq x$, using as criteria the uncertain cost $c(\omega)$. A detailed study of this problem has been recently done in \cite{BuchheimHenkeIrmai2022Knapsack}. It is shown that the set of chamber vertices is
\begin{equation}\label{eq:VerticesKnapsack}
	\mathscr{V}(D) = \left\{ {\textstyle\sum_{i\in I} a_i\ :\ I\subseteq [n]}  \right\}\cap [L,U],
\end{equation}
under the convention $\sum_{\emptyset}a_i = 0$. Choosing $a\in\R^{n}_+$ adequately, it is possible to obtain different values for every partial sum in \eqref{eq:VerticesKnapsack} (for instance, with $a_i = 2^i$ for all $i\in[n]$). This leads to $|\mathscr{V}(D)| = 2^{n}$, which is exponentially large with respect to the $2n+3$ constraints needed to define $D$. Furthermore, 
$\mathscr{K}(D)$ is given by the intervals between consecutive vertices, and so  $|\mathscr{K}(D)| = 2^{n}-1$. Thus, this is an example where $\mathscr{K}(D)$ is almost as large as $\mathscr{V}(D)$, and both are exponentially large with respect to the data of the problem.

When $a \in \N^n$ (i.e., the weights $(a_1,\ldots,a_n)$ are integers), one can give an overestimation of 
\AS{$\mathscr{V}(D)$} given by 
$\mathscr{V}(D) \subseteq\{0,1,2,\ldots,A\}\cap[L,U]$, where $A = \sum_{i=1}^n a_i$.
Thus, assuming that we have an oracle to evaluate the function $\theta(x) = \E[\langle d_2, y(x,\omega)\rangle  ]-\langle d_1,x\rangle$ for any $x\in [0,A]$ (or that we are solving a sample approximation), then we can solve $BK(a)$ in pseudo-polynomial time by computing the values $\{ \theta(0), \theta(1), \ldots, \theta(A) \}$ (the number $A$ requires $\log(A)$ bits to be represented).


\begin{example}\label{ex:Knapsack2D}
	Let us consider a very simple instance of problem \eqref{eq:knapsack}, where $n = 2$, $a_1 = 1$, $a_2 =2$, and $c$ follows a uniform distribution over $\Sph_2$. The polytope $D$ corresponding to the feasible region of the high-point relaxation of the problem is depicted in Figure \ref{fig:BK}.  
	
	In this case, the chamber vertices are exactly $\mathscr{V}(D) = \{0,1,2,3\}$, and so, the full-dimensional chambers are given by $\mathscr{K}(D) = \{ [0,1], [1,2], [2,3] \}$. The fiber map $S$ in each full-dimensional chamber, together with the normal cones of the extreme points, is depicted in Figure \ref{fig:SlicesAndNormalFans-Knapsack}.
		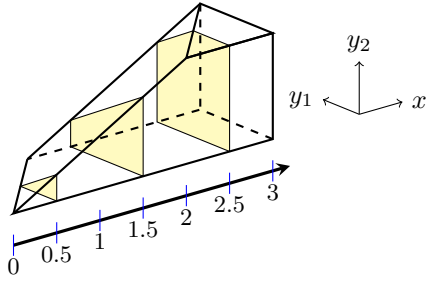
\begin{figure}[t]
		\begin{center}
			\tdplotsetmaincoords{70}{140}
			\begin{tikzpicture}[scale=1.5,
				tdplot_main_coords,
				axis/.style={-stealth,very thick}
				]
				 \filldraw[fill = yellow!30!white] (-.5,0,0)--(-.5,-.5,0)--(-.5,0,.25)--cycle;
				\filldraw[fill = yellow!30!white] (-1.5,0,0)--(-1.5,-1,0)--(-1.5,-1,.25)--(-1.5,0,.75) --cycle;
				\filldraw[fill = yellow!30!white] (-2.5,0,0)--(-2.5,-1,0)--(-2.5,-1,.75)--(-2.5,-0.5,1) --(-2.5,0,1)--cycle;
				
				\draw[thick,dashed] (-3,0,0)--(-3,-1,0)--(-1,-1,0);
				\draw[thick] (-1,-1,0) -- (0,0,0);
				\draw[thick] (0,0,0)--(-2,0,1)--(-3,0,1)--(-3,0,0)--cycle;
				\draw[thick] (-2,0,1)--(-3,-1,1)--(-3,0,1)--cycle;
				\draw[thick] (-3,-1,1)--(-1,-1,0)--cycle;
				\draw[dashed,thick] (-3,-1,1)--(-3,-1,0);
				
                \draw[axis] (0,0,-0.3)--(-3.2,0,-0.3);
                \foreach \x in {0,0.5,...,3}{\draw[thin]  [color=blue](-\x,0,-0.2)--(-\x,0,-0.4); \node at (-\x,0,-0.5) {\small \x}; }

				\draw[->] (-4,0,0) -- (-4.5,0,0) node[right]  {$x$};
				\draw[->] (-4,0,0) -- (-4,-0.5,0) node[left]  {$y_1$};
				\draw[->] (-4,0,0) -- (-4,0,0.5) node[above]  {$y_2$};
				\node at (1,0,0){};			
			\end{tikzpicture}
			\captionof{figure}{Polytope $D$. In yellow, $S(0.5)$, $S(1.5)$ and $S(2.5)$.}
			\label{fig:BK}
		\end{center}
	\end{figure}
		\begin{figure}[t]
		\begin{minipage}{0.3\textwidth}
			\centering
			\begin{tikzpicture}[scale = 1.3]
				\filldraw[fill = yellow!30!white] (0,0)--(.5,0)--(0,.25)--cycle;
				\node at (0,1.4){};
				
				\draw (-0.4,0.25) coordinate (c1)
				-- (0,.25) coordinate (b1)
				-- (0.4/2.24, 0.8/2.24+ 0.25) coordinate (a1)
				pic["\tiny $\tfrac{\pi}{2} + \alpha$", draw=orange, <->, angle eccentricity=1.2, angle radius=0.6cm]
				{angle=a1--b1--c1};
				\draw (0.4/2.24 +0.5, 0.8/2.24) coordinate (c2)
				-- (0.5,0) coordinate (b2)
				-- (0.5,-0.4) coordinate (a2)
				pic["\tiny $\pi - \alpha$", draw=orange, <->, angle eccentricity=1.2, angle radius=0.6cm]
				{angle=a2--b2--c2};
				\draw (0,-.4) coordinate (c3)
				-- (-0,0) coordinate (b3)
				-- (-0.4,0) coordinate (a3)
				pic["\tiny $\tfrac{\pi}{2}$", draw=orange, <->, angle eccentricity=1.2, angle radius=0.6cm]
				{angle=a3--b3--c3};
			\end{tikzpicture}
		\caption*{(a)}
		\end{minipage}
		\centering
		\begin{minipage}{0.3\textwidth}
		\centering
		\begin{tikzpicture}[scale = 1.3]
				\draw[fill = yellow!30!white] (0,0)--(1,0)--(1,.25)--(0,.75) --cycle;
				\node at (0,1.4){};
				
				\draw (-0.4,0.75) coordinate (c1)
				-- (0,.75) coordinate (b1)
				-- (0.4/2.24, 0.8/2.24+ 0.75) coordinate (a1)
				pic["\tiny $\tfrac{\pi}{2} + \alpha$", draw=orange, <->, angle eccentricity=1.2, angle radius=0.6cm]
				{angle=a1--b1--c1};
				\draw (0.4/2.24 +1, 0.8/2.24 + 0.25) coordinate (c2)
				-- (1,0.25) coordinate (b2)
				-- (1+0.4,0.25) coordinate (a2)
				pic["\tiny $\tfrac{\pi}{2} - \alpha$", draw=orange, <->, angle eccentricity=1.2, angle radius=0.6cm]
				{angle=a2--b2--c2};
				
				\draw (1.4,0) coordinate (c3)
				-- (1,0) coordinate (b3)
				-- (1,-0.4) coordinate (a3)
				pic["\tiny $\tfrac{\pi}{2}$", draw=orange, <->, angle eccentricity=1.2, angle radius=0.6cm]
				{angle=a3--b3--c3};
				
				\draw (0,-.4) coordinate (c4)
				-- (0,0) coordinate (b4)
				-- (-0.4,0) coordinate (a4)
				pic["\tiny $\tfrac{\pi}{2}$", draw=orange, <->, angle eccentricity=1.2, angle radius=0.6cm]
				{angle=a4--b4--c4};
		\end{tikzpicture}
	\caption*{(b)}
		\end{minipage}
		\begin{minipage}{0.3\textwidth}
		\centering
		\begin{tikzpicture}[scale = 1.3]
				\draw[fill = yellow!30!white] (0,0)--(1,0)--(1,.75)--(0.5,1) --(0,1)--cycle;

				\draw (0.5,1+0.4) coordinate (c1)
				-- (0.5,1) coordinate (b1)
				-- (0.4/2.24 + 0.5, 0.8/2.24+ 1) coordinate (a1)
				pic["\tiny $\alpha$", draw=orange, <->, angle eccentricity=1.2, angle radius=0.6cm]
				{angle=a1--b1--c1};
				\draw (0.4/2.24 +1, 0.8/2.24 + 0.75) coordinate (c2)
				-- (1,0.75) coordinate (b2)
				-- (1+0.4,0.75) coordinate (a2)
				pic["\tiny $\tfrac{\pi}{2} - \alpha$", draw=orange, <->, angle eccentricity=1.2, angle radius=0.6cm]
				{angle=a2--b2--c2};
				
				\draw (1.4,0) coordinate (c3)
				-- (1,0) coordinate (b3)
				-- (1,-0.4) coordinate (a3)
				pic["\tiny $\tfrac{\pi}{2}$", draw=orange, <->, angle eccentricity=1.2, angle radius=0.6cm]
				{angle=a3--b3--c3};
				
				\draw (-0.4,1) coordinate (c4)
				-- (0,1) coordinate (b4)
				-- (0,1+0.4) coordinate (a4)
				pic["\tiny $\tfrac{\pi}{2}$", draw=orange, <->, angle eccentricity=1.2, angle radius=0.6cm]
				{angle=a4--b4--c4};
				
				\draw (0,-.4) coordinate (c4)
				-- (-0,0) coordinate (b4)
				-- (-0.4,0) coordinate (a4)
				pic["\tiny $\tfrac{\pi}{2}$", draw=orange, <->, angle eccentricity=1.2, angle radius=0.6cm]
				{angle=a4--b4--c4};
		\end{tikzpicture}
	\caption*{(c)}
		\end{minipage}
	\caption{Slices $S(0.5)$, $S(1.5)$ and $S(2.5)$, and normal cones. $\alpha = \arccos(1/\sqrt{5})$.}\label{fig:SlicesAndNormalFans-Knapsack}
	\end{figure}
	
	Setting $\alpha = \arccos(1/\sqrt{5})$, it is possible to show that the centroid map $\mathfrak{b}$ of the fiber map $S$ is given by
	\[
		\mathfrak{b}(x) = \begin{cases}
			{\small\left(\frac{1}{4} + \frac{\alpha}{2\pi}\right) \begin{bmatrix}
		0\\ x/2
	\end{bmatrix} + \left(\frac{1}{4} + \frac{\pi/2 - \alpha}{2\pi}\right) \begin{bmatrix}
		x\\ 0 
	\end{bmatrix}}\quad&\text{ if }x\in [0,1],\\
	{\small \left(\frac{1}{4} + \frac{\alpha}{2\pi}\right) \begin{bmatrix}
		0\\ x/2
	\end{bmatrix} + \frac{1}{4}\begin{bmatrix}
		1\\ 0 
	\end{bmatrix}  + \frac{\pi/2 - \alpha}{2\pi} \begin{bmatrix}
		1\\ (x-1)/2 
	\end{bmatrix}} &\text{ if }x\in[1,2],\\
{\small\frac{1}{4}\begin{bmatrix}
	0\\ 1
\end{bmatrix} + \frac{\alpha}{2\pi} \begin{bmatrix}
	x-2\\ 1
\end{bmatrix} + \frac{1}{4}\begin{bmatrix}
	1\\ 0 
\end{bmatrix}  + \frac{\pi/2 - \alpha}{2\pi} \begin{bmatrix}
	1\\ (x-1)/2 
\end{bmatrix}}& \text{ if }x\in[2,3].
\end{cases}	
	\]
	This computation allows us to explicitly determine $\theta(x)$, simply by noting that $\theta(x) =  d^{\top} \mathfrak{b}(x) - \delta x$. 
 For appropriate values of $\delta$ and $d$, it is possible to set the minimum of $\theta$ to be attained at any of the chamber vertices.\hfill$\Diamond$
\end{example}


\section{Enumeration algorithm \label{sec:Algorithms}}

The rest of the work is focused on algorithms to solve problem \eqref{eq:target-problem-Bayesian} in the case when we can evaluate the objective function $x\mapsto\langle d_1,x\rangle + \E_{\beta_x}[\langle d_2,\cdot\rangle]$, or to solve problem \eqref{eq:target-problem-SAA}, otherwise. The key observation is that, thanks to Theorems \ref{thm:piecewise-Bayesian} and \ref{thm:piecewise-SAA}, both problems have the form
\begin{equation}\label{eq:target-problem-generic}
    \min_{x\in X} \theta(x),
\end{equation}
where $\theta: X\to \R$ is a continuous function, piecewise linear over the chamber complex $\mathscr{C}(D)$. Thus, we will provide algorithms to solve this generic problem. 

In this setting, Corollary \ref{cor:PiecewiseLinearHasSolution} gives us a natural strategy to solve problem \eqref{eq:target-problem-generic}. Compute the chamber vertices $\mathscr{V}(D)$ and evaluate the corresponding objective function $\theta$ at each one of them. In this section, we provide an enumeration algorithm to compute $\mathscr{V}(D)$ by sequentially solving mixed-integer programming problems which are formulated using $\mathscr{F}_{\leq\nx}$.

\begin{remark}\label{remark:V}
    Computing $\mathscr{V}(D)$ is at least as hard as computing all vertices of a polytope. Indeed, given an arbitrary (full-dimensional) polytope $P\subseteq \mathbb{R}^n$, one can consider 
    \[
    D:=\{(x,y)\in \mathbb{R}^{n+1} \, :\, x\in P,\, 0\leq y \leq 1\} = P\times[0,1].
    \]
    We observe that $\mathscr{V}(D)$ corresponds exactly to the extreme points of $P$. This follows since all faces of $D$ are either parallel to $P$ or orthogonal to $P$. In both cases, the projection of faces of $D$ are also faces of $P$ (see Figure \ref{fig:Lifting}). 
    \begin{figure}[t]
    \centering
			\tdplotsetmaincoords{75}{150}
			\begin{tikzpicture}[scale=1,
				tdplot_main_coords,
				axis/.style={->,thick},
				vector/.style={-stealth,very thick}, 
				vector guide/.style={dashed,red,thick},
				grid/.style={very thin,gray}
				]
				
				\coordinate (O) at (0,0,0);
				
				\fill[color=black!10!white] (-2,0,0)--(-1,-1,0)--(1,-1,0)
				--(2,0,0)--(1,1,0)--(-1,1,0)--cycle;
				\draw (2,0,0)--(1,1,0)--(-1,1,0)--(-2,0,0);
				\draw[dashed] (-2,0,0)--(-1,-1,0)--(1,-1,0)--(2,0,0);
				\node at (O) {$P$};
				\draw (-2,0,1)--(-1,-1,1)--(1,-1,1)
				--(2,0,1)--(1,1,1)--(-1,1,1)--cycle;
				\draw (2,0,0) -- (2,0,1);
				\draw (1,1,0) -- (1,1,1);
				\draw (-1,1,0) -- (-1,1,1); 
				\draw (-2,0,0) -- (-2,0,1); 
				\draw[dashed] (-1,-1,0) -- (-1,-1,1);
				\draw[dashed] (1,-1,0) -- (1,-1,1);
			\end{tikzpicture}
			\caption{Illustration of $D$ (the volume) as direct lifting of $P$ (in gray).}\label{fig:Lifting}
    \end{figure}
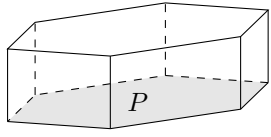
    
    To the best of our knowledge, the complexity of finding all vertices of a polytope $P$ is currently unknown. However, for a polyhedron $P$ (not necessarily bounded) it is known that it is NP-complete to decide, given a subset of vertices of $P$, if there is a new vertex of $P$ to add to the subset \cite{khachiyan2009generating}.
\end{remark}


\subsection{Mixed-integer formulation}

In order to compute chambers, Definition \ref{def:Chambers} tells us that we need access to the Face Lattice $\mathscr{F}$. However, Proposition \ref{prop:RequiredFaces} improves this requirement reducing $\mathscr{F}$ only to $\mathscr{F}_{\leq\nx}$. Recalling the representation \eqref{eq:FacesToSets}, we have that each face $F\in\mathscr{F}_{n}$ is represented by a set of constraints $J\subseteq [m]$ with $|J| = \codim(F) = (\nx+\ny) - n$. Thus,
\begin{equation}
|\mathscr{F}_{\leq\nx}| = \sum_{n = 0}^{\nx} |\mathscr{F}_n| \leq \sum_{n=\ny}^{\nx+\ny} {m\choose n} = O(m^{\nx+\ny}).
\end{equation}

Given a chamber $C\in\mathscr{C}(D)$ we define its \emph{label} as
\begin{equation}\label{eq:DefLabelSet}
\ell(C):=\left\{ F\in \mathscr{F}_{\leq\nx}\ :\ C\subseteq \pi(F) \right\}
\end{equation}
and similarly for a point $x\in X$ we define $\ell(x):=\ell(\sigma(x))$. Labels are a one-to-one representation of chambers, by noting that for every chamber $C\in\mathscr{C}(D)$, we can write $C = \bigcap_{F\in\ell(C)} \pi(F)$.

\begin{lemma}\label{lemma:maximallabels} For every $C_1,C_2\in \mathscr{C}(D)$ one has that
	\begin{equation}\label{eq:labelsubset}
	C_1\subseteq C_2 \iff \ell(C_1)\supseteq \ell(C_2).
	\end{equation}
	Thus, the set $\{ \ell(v)\ :\ v\in \mathscr{V}(D) \}$ corresponds to the maximal elements of $\{ \ell(C)\ :\ C\in \mathscr{C}(D) \}$, inclusion-wise.
\end{lemma}
\begin{proof}
The implication $\Rightarrow$ follows easily from the definition. For the $\Leftarrow$ implication if $\ell(C_1)\supseteq \ell(C_2)$, then
\[
C_1=\bigcap_{F\in \ell(C_1)} \pi(F)\subseteq \bigcap_{F\in \ell(C_2)} \pi(F)=C_2.
\]
The final part is direct from \eqref{eq:labelsubset}\qed
\end{proof}

The next proposition shows how to compute the label of a point $x\in X$, without the need of computing the projections of faces. We do this using a single LP formulation.
\begin{proposition}\label{prop:ComputingChambers} 
Let $x\in X$ and consider an optimal solution $(y_F^*\ : \ F\in \mathscr{F}_{\leq\nx})$ of the following LP problem 
	\begin{equation}\label{eq:PLForFindingLabel}
		\left\{\begin{array}{cl}
			\displaystyle\max_{(y_F)}\,&  \sum_{F\in \mathscr{F}_{\leq\nx}} g_F(x,y_F)\\
			s.t. & By_F \leq b - Ax,\quad\forall F\in \mathscr{F}_{\leq\nx}.
		\end{array}\right.
	\end{equation}	
Then
\begin{equation}
	\ell(x) = \left\{ F\in  \mathscr{F}_{\leq\nx}\ :\ g_F(x,y_F^*) = 0 \right\}.	
\end{equation}
\end{proposition}
\begin{proof}
By Proposition \ref{prop:RequiredFaces}, $x\in \ri(\sigma(x))$ and then for each $F\in\mathscr{F}_{\leq\nx}$ we have
\begin{align*}
        \sigma(x)\subseteq \pi(F) &\iff x\in\pi(F)\\
        &\iff \exists y^*_F\in S(x) \text{ such that }g_F(x,y_F^*) = 0
\end{align*}
where the last equivalence comes from \eqref{eq:FacesToSets}.
Since $x$ is fixed, and hence there are no coupling constraints among the variables $y_F$, $F\in \mathscr{F}_{\leq\nx}$, it follows that

any solution $(y_F^*\ : \ F\in \mathscr{F}_{\leq\nx})$ of \eqref{eq:PLForFindingLabel} must satisfy that $g_F(x,y_F^*) = 0$ for all faces in $\ell(x)$, and only for those faces.\qed
\end{proof}

Using these results, we can generate an element of $\mathscr{V}(D)$ through 
a mixed-integer linear formulation that aims at finding a maximal element of $\{ \ell(C)\ :\ C\in \mathscr{C}(D) \}$.
The following formulation achieves this: 
\begin{subequations}\label{eq:exactformulation}
\begin{align}
\max_{z,x,y} \quad & \sum_{F\in \mathscr{F}_{\leq\nx}} z_F \\
\mbox{s.t.} \quad & Ax + By_F \leq b && \forall F\in \mathscr{F}_{\leq\nx} \\
& A_Fx + B_Fy_F \geq b_F - M_F (1-z_F) && \forall F\in \mathscr{F}_{\leq\nx} \\ 
& z_F \in \{0,1\} && \forall F\in \mathscr{F}_{\leq\nx} 
\end{align}
\end{subequations}
Here, $y$ and $z$ stand for the vectors $(y_F\ :\ F\in\mathscr{F}_{\leq \nx})$ and $(z_F\ :\ F\in\mathscr{F}_{\leq \nx})$, respectively. For each $F\in \mathscr{F}_{\leq \nx}$, $A_F$, $B_F$ and $b_F$ are submatrices of $A$, $B$ and $b$ such that $F =\{ (x,y)\in D\ :\ A_Fx + B_Fy = b_F \}$, $M$ is a vector of $m$ positive values such that $A_{i} x + B_{i} y - b_i \geq -M_i$, for all $(x,y)\in D$, and $M_F$ the corresponding subvector, matching the indices of $b_F$. The vector $M$ is well defined when $D$ is a polytope, and can be easily computed using $m$ linear programs.
Formulation \eqref{eq:exactformulation} tries to ``activate'' (with $z_F=1$) as many
faces as possible keeping non-empty the intersection of their projection.

Let $(z^*, x^*, y^*)$ be an optimal solution of \eqref{eq:exactformulation}. Clearly,
$x^*$ is an element of $\mathscr{V}(D)$, thus, we can collect it
and focus on generating a new element of $\mathscr{V}(D)$. Noting that $\ell(x^*) =\{F \in \mathscr{F}_{\leq\nx}\, :\, z^*_F = 1\}$, we see that such a new element can be obtained by adding the following inequality constraint to
\eqref{eq:exactformulation}:
\begin{equation}\label{eq:nogoodcut}
    \sum_{F\in \mathscr{F}_{\leq\nx}\, :\, z_F^* = 0} z_F \geq 1.
\end{equation}
Since $(z^*, x^*, y^*)$ induced a maximal element of 
$\{ \ell(C)\ :\ C\in \mathscr{C}(D) \}$, we can easily see that constraint \eqref{eq:nogoodcut}
is removing only $x^*$ of $\mathscr{V}(D)$ from \eqref{eq:exactformulation}.

This procedure can be iterated until the optimization problem becomes infeasible; however, in order to avoid detecting infeasibility in our
computational implementation, we add a new binary variable $s$ that can relax
\eqref{eq:nogoodcut} when needed. Under these considerations, we present the precise model we use. Suppose we have generated a set $V\subseteq \mathscr{V}(D)$. We may generate an element of $\mathscr{V}(D)\setminus V$ or determine that $V= \mathscr{V}(D)$ using the following optimization
problem:
\begin{subequations}\label{eq:exactformulation_final}
\begin{align}
\max_{z,s,x,y} \quad & \sum_{F\in \mathscr{F}_{\leq\nx}} z_F \\
\mbox{s.t.} \quad & Ax + By_F \leq b && \forall F\in \mathscr{F}_{\leq\nx} \label{con:D} \\
&  A_Fx + B_Fy_F \geq b_F - M_F (1-z_F) && \forall F\in \mathscr{F}_{\leq\nx} \label{con:bigM} \\
& s+\sum_{F\not\in \ell(v)} z_F \geq 1  && \forall v\in V \label{con:slb} \\
& \sum_{F\in \mathscr{F}_{\leq\nx}} z_F \leq |\mathscr{F}_{\leq\nx}|(1-s)  \label{con:sub} \\
& z_F \in \{0,1\} && \forall F\in \mathscr{F}_{\leq\nx} \\
&s \in \{0,1\}
\end{align}
\end{subequations}

\begin{lemma}\label{lem:CorrectnessEnumerative}
    Problem \eqref{eq:exactformulation_final} is feasible provided that $D\neq \emptyset$. Moreover, in an optimal solution 
    \((z^*,s^*,x^*,y^*)\), then one (and only one) of the following hold
    \begin{itemize}
    \item $s^* = 0$ and $x^*\in \mathscr{V}(D)\setminus V$.
    \item $s^* = 1$ and $\mathscr{V}(D) = V$.
    \end{itemize}
\end{lemma}
\begin{proof}
Let us first show that \eqref{eq:exactformulation_final} is feasible. Since $D$ is nonempty, there exists a point $(\bar{x},\bar{y})\in D$. Then, we can consider $\hat{s} = 1$, $\hat{x} = \bar{x}$, and, for each $F\in \mathscr{F}_{\leq\nx}$, $\hat{z}_F = 0$ and $\hat{y}_F = \bar{y}$. It is easy to verify that  $(\hat{z},\hat{s},\hat{x},\hat{y})$ is feasible.

For the second part of the statement, we begin by tackling the case of $s^* = 0$. By constraint \eqref{con:slb}, we know that $x^*\not\in V$ since, for each $v\in V$, $z_F^* = 1$ for some $F\not\in \ell(v)$ (recall that $\ell(v)$ is the 
set of all faces whose intersection of projections is $\{v\}$). The fact that $x^*\in \mathscr{V}(D)$
follows from the optimality of the solution and Lemma \ref{lemma:maximallabels}, as constraint \eqref{con:slb} only removes elements of $V$.

Now assume $s^* = 1$. Constraint \eqref{con:sub} implies that $z^*_F = 0$, $\forall F\in\mathscr{F}_{\leq\nx}$. If there were an element $\tilde{x}\in \mathscr{V}(D)\setminus V$, we could clearly set $\tilde{z}_F = 1$ $\forall F \in \ell(\tilde{x})$ and find $\tilde{y}_F$ according to \eqref{eq:PLForFindingLabel}. The vector $(\tilde{z}, 0,\tilde{x},\tilde{y}_F)$ would be feasible and, since $\ell(\tilde{x})$ must be nonempty,
\[\sum_{F\in \mathscr{F}_{\leq\nx}} z^*_F = 0 < \sum_{F\in \mathscr{F}_{\leq\nx}} \tilde{z}_F  \]
This contradicts the optimality of \((z^*,s^*,x^*,y^*)\), thus $\mathscr{V}(D) = V$.\qed
\end{proof}
In Algorithm \ref{alg:Enumeration} we formalize our enumeration procedure. The correctness of the algorithm is given by Lemma \ref{lem:CorrectnessEnumerative}. To solve problem \eqref{eq:target-problem-generic}, it is enough to run Algorithm \ref{alg:Enumeration}, and evaluate $\theta$ over the set $V = \mathscr{V}(D)$.

\begin{algorithm}[t]
	\SetAlgoLined
        \text{\bf Input:} $A,B,b$ defining a polytope $D=\{(x,y)\in \mathbb{R}^{n_x} \times \mathbb{R}^{n_y}\,:\, Ax+By\leq b\}$\;
	Set $V = \emptyset$, $s^*= 0$\;
        Compute $\mathscr{F}_{\leq\nx}(D)$\;
	\While{true}{
            Solve problem \eqref{eq:exactformulation_final} and obtain an optimal solution \((z^*,s^*,x^*,y^*)\)\;
            \uIf{$s^*=0$}{
                $V\leftarrow V\cup \{x^*\}$\;
                Store $\ell(x^*) =\{F \in \mathscr{F}_{\leq\nx}(D)\, :\, z^*_F = 1\}$\;
            }
            \Else{
                \emph{break}\;
            }
	}
	\bigskip
	\KwResult{The set $V=\mathscr{V}(D)$}
	\caption{Chamber vertex enumeration algorithm}\label{alg:Enumeration}
\end{algorithm}


\AS{\begin{remark}
    It is worth noting that Algorithm \ref{alg:Enumeration} is double-exponential, since problem \eqref{eq:exactformulation_final} is a MIP whose number of variables is linear on the size of the Face Lattice, this last being possibly exponential on the size of the original problem. This makes Algorithm \ref{alg:Enumeration} far from being scalable. It is presented, however, to give a computational baseline (useful for really small problems), since, to the best of our knowledge, there is no other general algorithm in the literature computing global solution for problems of the form \eqref{eq:target-problem-generic} (and in particular \eqref{eq:target-problem}).
\end{remark}}

\subsection{Practical computational considerations}

Algorithm \ref{alg:Enumeration} is basically made up of two (expensive)
computations: the (partial) Face Lattice $\mathscr{F}_{\leq\nx}$ and solving problem \eqref{eq:exactformulation_final}.
For Face Lattice computations, we rely on Polymake \cite{polymake:2000}, which is a highly optimized package for these purposes. For solving problem \eqref{eq:exactformulation_final} we use Gurobi \cite{gurobi}, which we enhance using the following simple local search heuristic.

\paragraph{A local-search heuristic:} When solving problem \eqref{eq:exactformulation_final} in Algorithm \ref{alg:Enumeration}, it is not strictly necessary to have the true optimal solution. A feasible solution 
\((\hat{z},\hat{s},\hat{x},\hat{y})\) such that $\hat{z}$ induces a 
maximal element of $\{ \ell(C)\ :\ C\in \mathscr{C}(D) \}$ suffices.
We can exploit this in the following way: whenever a feasible solution
\((\hat{z},\hat{s},\hat{x},\hat{y})\) is found by Gurobi, we flip any component $\hat{z}_F = 0$ to 1, and check if the $x$ components can be modified to obtain a new feasible solution to \eqref{eq:exactformulation_final}. This is a linear program. We repeat the process for each $\hat{z}_F = 0$ and pass the resulting solution to Gurobi via a callback.


\section{Monte-Carlo approximation scheme 
	\label{sec:Monte-Carlo}}

The enumeration algorithm of Section \ref{sec:Algorithms} has several drawbacks. First, it requires to compute (in practice) the whole Face Lattice of $D$, which might depend exponentially on the whole dimension $\nx+\ny$. Even with the Face Lattice at hand, computing all chamber vertices in $\mathscr{V}(D)$ can be hard, and as we have seen in Section \ref{sec:Examples}, $\mathscr{V}(D)$ might be exponentially large.

Another approach, that we explore in this section, is  to try to compute the collection of full-dimensional chambers $\mathscr{K}(D)$. If one has access to this family, one could find the solution of problem \eqref{eq:target-problem-generic} as follows:
\begin{enumerate}
   \item For each $K\in \mathscr{K}(D)$, compute $d_K\in\R^{\nx}$ such that $\traza{\theta}{K} = \langle d_K,\cdot\rangle + a_K$ (constant $a_K$ can be disregarded), and solve the linear problem $\min_{x\in K}\langle d_K,x\rangle$.
    \item Noting that $X = \bigcup_{K\in \mathscr{K}(D)} K$, the solution $x^*$ with minimal value $\theta(x^*)$ among those found in the first step must be the solution of problem \eqref{eq:target-problem-generic}.
\end{enumerate}
Listing the whole family $\mathscr{K}(D)$ could be exponentially long, as the Knapsack problem illustrates in Section \ref{sec:Examples}, but we have an advantage: If we draw a point $x\in X$ randomly, $\sigma(x)$ will be a full-dimensional chamber almost surely.
Indeed, this follows from Proposition \ref{prop:RequiredFaces} and the fact that there are finitely many chambers in $\mathscr{C}(D)$, and only those in $\mathscr{K}(D)$ are not negligible.


\subsection{Label representation for full-dimensional chambers}

To simplify the exposition, from now on, we will assume that $X\subseteq [0,1]^{\nx}$ and we will write $X^c:=[0,1]^{\nx}\setminus X$. Since $D$ is a polytope, this requirement can be easily attained by a mild change of variables. To be able to consider samples in $[0,1]^{\nx}$ (which are easy to generate, and they are independent of the specific structure of $D$), we identify $\ell(X^c)$ with $\emptyset$, which correspond to the set of faces that are activated by a point $x\in X^c$. 

Note that whenever $F\in \ell(K)$ for a full-dimensional chamber $K\in\mathscr{K}(D)$, then $\dim(F) \geq \nx$. Then, we obtain the following direct lemma, that improves Proposition \ref{prop:RequiredFaces}.

\begin{lemma}\label{lemma:DimReduction} let $K \in \mathscr{K}(D)$. Then, for any $x\in\into(K)$, we have that
    \begin{equation}
        K = \bigcap\left\{ \pi(F)\ :\ F\in\mathscr{F}_{\nx},\, x\in \pi(F) \right\}\quad \text{and}\quad \ell(K)\subseteq \mathscr{F}_{\nx}.
    \end{equation}
\end{lemma}

This is already a substantial improvement since computing $\mathscr{F}_{\nx}$ is less demanding than computing $\mathscr{F}_{\leq\nx}$. Moreover, the size of $\mathscr{F}_{\nx}$ can be controlled by fixing only the follower's dimension $\ny$ since
\begin{equation}
    |\mathscr{F}_{\nx}| \leq \left|\{ J\subseteq [m]\ :\ |J| = \ny \}\right| = {m\choose\ny} = O(m^{\ny}).
\end{equation}

Motivated by this observation, for every $x\in [0,1]^{\nx}$ we define
\begin{equation}\label{eq:Reducedlabel}
    \hat{\ell}(x) = \{ F\in\mathscr{F}_{\nx}\ :\ x\in \pi(F) \}.
\end{equation}
Then, for almost every $x\in [0,1]^{\nx}$, one has that $\hat{\ell}(x) = \ell(x)$.
The next lemma, which is also direct, allows us to compute the linear component of $\theta$ in a full-dimensional chamber $K$ at any point $x\in\into(K)$.

\begin{lemma}\label{lemma:Compute-dK} Let $x\in X$ and let $\ell = \hat{\ell}(x)$. If $K:=\sigma(x)$ is a full-dimensional chamber, then for each $j\in [\nx]$, the following problem
\begin{equation}\label{eq:ProblemToFinddK}
\left\{\begin{array}{cl}
    \displaystyle\max_{t,\,(y_F)_{F\in \ell}} &  t  \\
    s.t. & \begin{array}{ll}(x+te_j,y_F) \in D,\quad&\forall F\in\ell,\\
   g_F(x+te_j,y_F) =0,\quad&\forall F\in\ell,
   \end{array}
\end{array}   \right. 
\end{equation}
has a solution $t_j^*>0$. Moreover, for every  function $\theta:X\to \R$ continuous and piecewise linear over the chamber complex $\mathscr{C}(D)$, the vector $d_{\ell}\in \R^{\nx}$ such that $\traza{\theta}{K} = \langle d_\ell,\cdot\rangle + a_\ell$ (for some $a_\ell\in\R$) can be taken as
\[
d_{\ell} = \begin{pmatrix}
\frac{\theta(x + t_1^*e_1) - \theta(x)}{t_1^*},&
\ldots,&
\frac{\theta(x + t_{\nx}^*e_{\nx}) - \theta(x)}{t_{\nx}^*}
\end{pmatrix}^{\top}.
\]
\end{lemma}
\begin{proof}
Let $K = \sigma(x)$. Since $x\in \ri(K) = \into(K)$, it is clear that for every $j\in[\nx]$, the problem
\[
\max_{t}\left\{ t\ :\ x+te_j\in K\right\},
\]
has a solution $t_j^*>0$. The conclusion then follows by noting that the above optimization problem is equivalent to problem \eqref{eq:ProblemToFinddK}. The rest is direct.\qed
\end{proof}

Based on both lemmas, we propose the Monte-Carlo algorithm presented in Algorithm \ref{alg:MonteCarlo} to approximate the solution of problem \eqref{eq:target-problem-generic}, by randomly drawing points from $[0,1]^{\nx}$. The algorithm returns two lists that allow to compute the visited chambers during the execution, and the numbers of visits per chamber. This information is used to estimate error measures, as it is described in the next subsection.

\begin{algorithm}[t]
	\SetAlgoLined
	\text{\bf Input:} $A,B,b$ defining a polytope $D=\{(x,y)\in \mathbb{R}^{n_x} \times \mathbb{R}^{n_y}\,:\, Ax+By\leq b\}$, a function $\theta$ continuous, computable, and piecewise linear over $\mathscr{C}(D)$\;
	Generate a (uniformly iid) training sample $S$ of size $N$ over $[0,1]^{\nx}$\;
	Set $\texttt{List1} = \emptyset$, $\texttt{List2} = \emptyset$, $\hat{x}=NaN$, $\hat{\theta} = \infty$\;
	\ForEach{$\xi\in S$}{
        		Compute $\ell = \hat{\ell}(\xi):= \{ F\in \mathscr{F}_{\nx}\ :\ \xi\in \pi(F) \}$\;
        		$\texttt{List1}\gets \texttt{List1}\cup\{(\xi,\ell)\}$\;
        		\uIf{$\ell\in \texttt{List2}$ or $\ell = \emptyset$}{
        		\textbf{continue}\;}
        		
        		\Else{
        		$\texttt{List2}\gets \texttt{List2}\cup\{\ell\}$\;
        		Compute $d_{\ell}$ as in Lemma \ref{lemma:Compute-dK}\;
        		Solve the problem
        		\[
        		\left\{\begin{array}{cl}
		\displaystyle\min_{x,(y_F)_{F\in\ell}}\,& \langle d_{\ell},x\rangle\\
		s.t. & \begin{array}{ll}
			Ax + By_F \leq b\quad&\forall F\in \ell\\
			g_F(x,y_F) = 0\quad&\forall F\in \ell
		\end{array}
	\end{array}\right.
    \]
    finding a solution $\hat{x}_{\ell}$ and set the value $\hat{\theta}_{\ell} = \theta(\hat{x}_{\ell})$\;
	\If{$\hat{\theta}_{\ell} < \hat{\theta}$}{ $\hat{x}\gets \hat{x}_{\ell}$, $\hat{\theta} \gets \hat{\theta}_{\ell}$\;}
	}
	}
	\bigskip
	\KwResult{The pair solution-value $(\hat{x},\hat{\theta})$, 
 and the lists \texttt{List1}, \texttt{List2}.}
	\caption{Monte-Carlo algorithm}\label{alg:MonteCarlo}
\end{algorithm}

\subsection{Volume-error estimators}

The main drawback of the 
Algorithm \ref{alg:MonteCarlo} is that we do not know if the result is an optimal solution of problem \eqref{eq:target-problem-generic} or not. 
Therefore, we would like to provide a suitable measure of error for the estimation $(\hat{x},\hat{\theta})$ produced by Algorithm \ref{alg:MonteCarlo}. Error bounds based on Lipschitzness of $\theta$ are not an option due to the curse of dimensionality: the number of samples required to be (metrically) close to the optimal solution grows exponentially with $\nx$  (see, e.g., \cite{Leobacher2014introduction}). Instead, we propose to consider the volume of the ``unseen chambers'' during the execution of the Monte-Carlo algorithm.

To do so, define $\mathscr{L} = \mathcal{P}(\mathscr{F}_{\nx})$, where $\mathcal{P}(\cdot)$ denotes the power set. Recalling that $\ell(X^c)=\emptyset$, we get that $\{ \ell(K)\ :\ K\in \mathscr{K}(D)\cup\{X^c\} \}$ is a subset of $\mathscr{L}$.  For each $\ell\in \mathscr{L}$, let us define $\ind_{\ell} := \ind_{\into(K_{\ell})}$, where
\begin{equation}\label{eq:ChamberOfLabel}
\begin{aligned}
K_{\ell} = \left\{\begin{array}{cl}
K\in \mathscr{K}(D)&\text{ if }\ell = \ell(K)\\
X^c&\text{ if }\ell = \emptyset.\\
\emptyset&\text{ otherwise.}
\end{array}\right.
\end{aligned}
\end{equation}
Let $\mathcal{N} = [0,1]^{\nx}\setminus \bigcup\{ \into(K)\ :\ K\in \mathscr{K}(D)\cup \{X^c\} \}$. Note that $\mathcal{L}^{\nx}(\mathcal{N}) = 0$, and that for each $x\in [0,1]^{\nx}\setminus \mathcal{N}$ there is one and only one $\ell\in\mathscr{L}$ such that $\ind_{\ell}(x) = 1$, which is given by $\ell = \hat{\ell}(x)$, as in equation \eqref{eq:Reducedlabel}.

In what follows, for a set $K\subseteq \R^{\nx}$ we write $\mathrm{Vol}(K)$ as its $\nx$-dimensional volume, i.e., $\mathrm{Vol}(K) = \mathcal{L}^{\nx}(K)$, where $\mathcal{L}^{\nx}$ stands for the Lebesgue measure over $\R^{\nx}$.

\begin{proposition}\label{prop:VarianceSumIndicators} Let $\{\ell_1,\ell_2,\ldots,\ell_p\}$ be a subset of $\mathscr{L}$ with no repeated elements. Let $Z$ be a uniformly distributed random vector over $[0,1]^{\nx}$. Then,
    \begin{enumerate}
        \item $\sum_{i=1}^p \ind_{\ell_i}$ coincides with the indicator function of $\bigcup_{i=1}^p\into(K_{\ell_i})$. 
        \item $\mathbb{E}\left[ \sum_{i=1}^p \ind_{\ell_i}(Z)\right] = \mathrm{Vol}\left( \bigcup_{i=1}^p K_{\ell_i} \right) = \sum_{i=1}^p \mathrm{Vol}\left(K_{\ell_i} \right)$.
        \item $\mathrm{Var}\left[ \sum_{i=1}^p \ind_{\ell_i}(Z)\right] \leq 1$.
    \end{enumerate}
\end{proposition}

\begin{proof} Set $U = \bigcup_{i=1}^p\into(K_{\ell_i})$. 
 Since $\{ \ri(C)\ :\ C\in\mathscr{C}(D)\}\cup\{X^c\}$ is a partition of $[0,1]^{\nx}$ and since the labels $\{\ell_1,\ell_2,\ldots,\ell_p\}$ are all different, we get that the union defining $U$ is pairwise disjoint. Then,
\begin{align*}
\ind_{U}(x) = 1 &\iff \exists! i\in[p], x\in \into(K_{\ell_i})\iff \sum_{i=1}^p \ind_{\ell_i}(x) = 1.
\end{align*}
This proves statement \textit{1}. Statement \textit{2} follows directly from \textit{1}, by noting that the uniform distribution of $Z$ entails that $\E[\ind_O(Z)] = \mathrm{Vol}(O)$ for each $O\subseteq[0,1]^{\nx}$.
Let us prove statement \textit{3}. By definition, we have $\mathrm{Vol}(U)\leq 1$ and thus $\max\{\mathrm{Vol}(U),1-~\mathrm{Vol}(U)\}\leq 1$. We can write
\begin{align*}
  \mathrm{Var}\left[ \sum_{i=1}^p \ind_{\ell_i}(Z)\right] &= \int_{[0,1]^{\nx}} (\ind_U(z) - \mathrm{Vol}(U))^2dz\\
  &\leq  \int_{[0,1]^{\nx}} \max\{\mathrm{Vol}(U),1-\mathrm{Vol}(U) \}^2dz\leq \int_{[0,1]^{\nx}} 1dz = 1. 
\end{align*}
This concludes the proof.\qed
\end{proof}


Now, let us consider a sample $S = \{ \xi_1,\ldots,\xi_{N} \}$ of independent uniformly distributed random variables over $[0,1]^{\nx}$. We define the random set $K(S)$ as
\[
K(S) := \bigcup_{i=1}^{N_1}  K_{\ell(\xi_i)}. 
\]
Then, the complement of $K(S)$, namely $K(S)^c = [0,1]^{\nx}\setminus K(S)$, can be represented with its indicator function, and can be written (almost surely) as
\[
\ind_{K(S)^c}(x) = \sum_{\ell\in \mathscr{L}} \left[\prod_{i=1}^{N} (1 - \ind_{\ell}(\xi_i))\right]\ind_{\ell}(x),\quad\forall x\in [0,1]^{\nx}\setminus \mathcal{N}.
\]
\begin{definition} For a sample $S$ of $[0,1]^{\nx}$, we define the \emph{volume-error estimator} as
    \begin{equation}
        \rho(S) = \mathrm{Vol}(K(S)^c) = 1- \mathrm{Vol}(K(S)).
    \end{equation}
\end{definition}


The volume-error estimator can be interpreted as follows: For a given sample $S$ and an optimal solution $x^*$ of problem \eqref{eq:target-problem-generic}, $\rho(S)$ is the probability of $x^*\in K(S)^c$, if $x^*$ were a random vector following a uniform distribution in $[0,1]^{\nx}$. Another interpretation is that $\rho(S)$ is the proportion of points that still have a chance to have a lower value than $\hat{\theta}$ since, by construction, $\hat{\theta}\leq \theta(x)$ for all $x\in K(S)$. 

To estimate $\rho(S)$, we will divide the sample $S$ in two subsamples: a first sample $S_1= \{\xi_1,\ldots,\xi_{N_1}\}$, called the \emph{training sample}, and a second one $S_2 = \{ \zeta_1,\ldots,\zeta_{N_2} \}$, called the \emph{testing sample}. It is important to address that this division must be decided before any realization of $S=(S_1,S_2)$ is obtained, to preserve independence. With this division, we can define the estimator
\begin{equation}\label{eq-def:DoubleEstimator}
U(S_1,S_2) = \frac{1}{N_2}\sum_{j=1}^{N_2} \ind_{K(S_1)^c}(\zeta_j)
= \frac{1}{N_2}\sum_{j=1}^{N_2}\sum_{\ell\in \mathscr{L}} \left[\prod_{i=1}^{N_1} (1 - \ind_{\ell}(\xi_i))\right]\ind_{\ell}(\zeta_j),
\end{equation}
which selects a random set $K(S_1)$ given by all the chambers observed by the training sample $S_1$, and then estimates the volume of the complement using  the testing sample $S_2$. Observe that, from the array \texttt{List1} produced by Algorithm \ref{alg:MonteCarlo}, we can easily compute $U(S_1,S_2)$ as
\[
U(S_1,S_2) = \frac{1}{N_2}\sum_{j=1}^{N_2} \left[\prod_{i=1}^{N_1} (1 - \ind_{\hat{\ell}(\zeta_j)}(\xi_i))\right] = \frac{1}{N_2}\sum_{j=1}^{N_2} \left[\prod_{i=1}^{N_1} \big(\hat{\ell}(\zeta_j)\neq \hat{\ell}(\xi_i)\big)\right],
\]
where $\big(\hat{\ell}(\zeta_j)\neq \hat{\ell}(\xi_i)\big)$ is interpreted as its associated boolean value ($1$ if the inequality holds, and $0$ otherwise). 


\begin{proposition}\label{prop:errEstimation} Let $(S_1,S_2)$ be a fixed training-testing division for the sample $S$. The following assertions hold:
\begin{enumerate}
    \item The conditional expectation $\mathbb{E}[U(S_1,S_2)|S_1]$ is well-defined, $S_1$-measurable and it verifies that
    \[
    \mathbb{E}[U(S_1,S_2)|S_1] = \mathrm{Vol}(K(S_1)^c) = 1-\mathrm{Vol}(K(S_1)).
    \]
    \item The expected error is bounded by $N_2^{-1/2}$, i.e.,
    \[
    \mathbb{E}\Big[ |e(S_1,S_2)| \Big]\leq \frac{1}{\sqrt{N_2}},
    \]
    where $e(S_1,S_2) = \mathrm{Vol}(K(S_1)^c) - U(S_1,S_2)$.
\end{enumerate}
\end{proposition}
\begin{proof}
Set $\Omega_1 = \big([0,1]^{n_x}\big)^{N_1}$ and $\Omega_2 = \big([0,1]^{n_x}\big)^{N_2}$. Let $\mathbb{P}_1$ be the Lebesgue probability measure over $\Omega_1$ and $\mathbb{P}_2$ be the Lebesgue probability measure over $\Omega_2$. By its construction in \eqref{eq-def:DoubleEstimator}, it is clear that $U$ is $(\mathbb{P}_1\times \mathbb{P}_2)$-integrable in the probability product space $(\Omega_1,\mathcal{B}(\Omega_1),\mathbb{P}_1)\times (\Omega_2,\mathcal{B}(\Omega_2),\mathbb{P}_2)$, and thus, Fubini's theorem (see, e.g., \cite[Theorem 14.19]{klenke2013probability}) entails that
\[
\mathbb{E}[U(S_1,S_2)|S_1] = \int U(S_1,S_2)d\mathbb{P}_2(S_2)
\]
is well-defined and $\mathbb{P}_1$-integrable. Now, for a fixed value of $S_1$, we get $K(S_1)^c$ is a fixed closed set and that $U(S_1,S_2) = \frac{1}{N_2}\sum_{j=1}^{N_2} \ind_{K(S_1)^c}(\zeta_j)$. Thus,
\[
\int U(S_1,S_2)d\mathbb{P}(S_2) 
= \frac{1}{N_2}\sum_{j=1}^{N_2} \mathbb{E}_{\zeta_j}[\ind_{K(S_1)^c}(\zeta_j)] = \mathrm{Vol}(K(S_1)^c).
\]
This finishes the first part of the proof. Now, for the second part, we first observe that 
\begin{equation}\label{eq:errDef}
e(S_1,S_2)^2 = (\mathbb{E}[U(S_1,S_2)|S_1] - U(S_1,S_2))^2,
\end{equation}
which is $(\mathbb{P}_1\times \mathbb{P}_2)$-integrable. Again, applying Fubini's theorem, we get that
\[
\mathbb{E}[ e(S_1,S_2)^2 ] = \mathbb{E}\Big[ \mathbb{E}[ e(S_1,S_2)^2 | S_1] \Big] = \mathbb{E}\Big[ \mathrm{Var}[U(S_1,S_2)|S_1] \Big].
\]
Finally, the conditional variance $\mathrm{Var}[U(S_1,S_2)|S_1]$ is just given by the variance of the mean estimator of $\ind_{K(S_1)^c}$ for the sample $S_2$, i.e.,
\begin{align*}
    \mathrm{Var}[U(S_1,S_2)|S_1] &= \mathbb{E}[U(S_1,S_2)^2|S_1] - \mathbb{E}[U(S_1,S_2)|S_1]^2\\
    &=\left(\frac{1}{N_2}\sum_{j=1}^{N_2} \ind_{K(S_1)^c}(\zeta_j)\right)^2 - \mathrm{Vol}(K(S_1)^c)^2.
\end{align*}
Assuming $K(S_1)^c$ as fixed, the above expression is bounded by $1/N_2$ (see, e.g., \cite[Chapter 1]{Leobacher2014introduction}), given that $\mathrm{Var}(\ind_{K(S)^c})\leq 1$ thanks to Proposition \ref{prop:VarianceSumIndicators}. Thus, by a mild application of Jensen's inequality, we get that
\begin{align*}
\mathbb{E}[ |e(S_1,S_2)| ] &\leq \sqrt{\mathbb{E}[ e(S_1,S_2)^2 ]}
= \left( \int \mathrm{Var}[U(S_1,S_2)|S_1] d\mathbb{P}(S_1) \right)^{-1/2} \leq 
\frac{1}{\sqrt{N_2}}.
\end{align*}
This finishes the proof. \qed
\end{proof}

For a large enough training sample $S_1$, we would expect $\mathrm{Vol}(K(S_1)^c)$ to be small with high probability. However, we don't know how small $\mathrm{Vol}(K(S_1)^c)$ actually is. The size of $\mathrm{Vol}(K(S_1)^c)$ is estimated with $S_2$.

\begin{theorem}\label{thm:ConfidenceInterval} For a fixed training-testing division $(S_1,S_2)$ of sample $S$ and for a confidence level $\varepsilon>0$, the volume-error estimator $\rho(S)$ is less than $U(S_1,S_2) + \frac{\varepsilon^{-1}}{\sqrt{N_2}}$, i.e.,
\begin{equation}\label{eq:ConfidenceInterval}
    \P\left[ \rho(S) \leq U(S_1,S_2) + \frac{\varepsilon^{-1}}{\sqrt{N_2}} \right] \geq 1-\varepsilon.
\end{equation}
\end{theorem}
\begin{proof}
We know that $\rho(S) = \mathrm{Vol}(K(S_1\cup S_2)^c) \leq \mathrm{Vol}(K(S_1)^c)$. Thus, as a mild application of Markov's inequality, we can write
\begin{align*}
  \P\left[ \rho(S) \leq U(S_1,S_2) + \frac{\varepsilon^{-1}}{\sqrt{N_2}}\right] &= \P\left[ \mathrm{Vol}(K(S_1)^c) \leq U(S_1,S_2) + \frac{\varepsilon^{-1}}{\sqrt{N_2}}\right]\\
  &\geq \P\left[ \big|\mathrm{Vol}(K(S_1)^c) - U(S_1,S_2)\big| \leq \frac{\varepsilon^{-1}}{\sqrt{N_2}}\right]\\
  &= 1- \P\left[ \big|e(S_1,S_2)\big| > \frac{\varepsilon^{-1}}{\sqrt{N_2}}\right]\\
  &\geq 
  1 - \varepsilon\sqrt{N_2}\E[|e(S_1,S_2)|].
\end{align*}
Since $\E[|e(S_1,S_2)|] \leq 1/\sqrt{N_2}$ by Proposition \ref{prop:errEstimation}, we deduce \eqref{eq:ConfidenceInterval}, finishing the proof.\qed
\end{proof}

\section{Numerical experiments and conclusion\label{sec:Numerical}}
The goal of this section is to illustrate how both Algorithms \ref{alg:Enumeration} and \ref{alg:MonteCarlo} perform in some instances of bilevel optimization with uncertain cost. To do so, we adapt some deterministic bilevel problems available in the literature.

\subsection{Set-up and instances}

We implemented both Algorithms \ref{alg:Enumeration} and \ref{alg:MonteCarlo} in Julia 1.8.2 \cite{bezanson2017julia}, using Polymake \cite{polymake:2000} to compute the faces of a polytope and Gurobi 9.5.2 \cite{gurobi} to solve \eqref{eq:exactformulation_final} and any auxiliary LP.
Our code is publicly available in \url{https://github.com/g-munoz/bilevelbayesian}.
All experiments were run single-threaded  on a Linux machine with an Intel Xeon Silver 4210 2.2G CPU and 128 GB RAM. The main objectives behind these experiments were (1) to determine how Algorithm \ref{alg:Enumeration} scales and (2) how well the Monte-Carlo algorithm performs in comparison to the exact method. A global time limit of 15 minutes was set for Algorithm \ref{alg:Enumeration}; in case this time limit is met, only the chamber vertices that were found are used.

We focus our attention on sample average formulations, as in \eqref{eq:target-problem-SAA}, where the lower-level cost is assumed to have a uniform distribution over the unit sphere. We use instances from two publicly available libraries: BOLib \cite{zhou2020bolib} and the bilevel instances in \cite{corallib}, which we call CoralLib. Since our approach relies on computing a (possibly exponentially) large number of faces, we can only consider low-dimensional instances at the moment: we restrict to $n_x + n_y \leq 10$.

Additionally, we consider randomly generated instances of the stochastic bilevel continuous knapsack problem \cite{BuchheimHenkeIrmai2022Knapsack}, as presented in Section \ref{sec:Examples}.
In our experiments, we consider $a$ to be a random non-negative vector, $\delta = 1/4$, and $d$ a vector of ones. We call \emph{Knapsack\_i} an instance generated for $n_y = i$. While these instances have a more efficient algorithm for them than the one presented here (see \cite{BuchheimHenkeIrmai2022Knapsack}), they are helpful in showing how well our general-purpose Monte-Carlo algorithm performs. 

In all experiments, we used a sample of size $N_0=100$ for the follower's cost vector. The same sample is used in both algorithms to better compare their performance.
Additionally, in Algorithm \ref{alg:MonteCarlo} we used a training sample of size $N_1=100$ and a testing sample of size $N_2=100$ as well. 

\subsection{Results}

In Table \ref{table:mainresults}, we compare the performance of both methods.

The  ``Obj gap'' column shows how far the value of the Monte-Carlo algorithm is from the exact method, i.e., if $\text{val}_i$ is the value obtained by Algorithm $i$, then the gap is 
$\mbox{Gap} = |\text{val}_1|^{-1}(\text{val}_2 - \text{val}_1)$.
Since we ran Algorithm \ref{alg:Enumeration} with a time limit, it may be that $\mbox{Gap} < 0$, which indicates the Monte-Carlo algorithm performing better than the exact method. 
 The ``Error" column shows the error estimation as per \eqref{eq-def:DoubleEstimator}. All execution times are in seconds.

The results in Table \ref{table:mainresults}
clearly show an advantage of the Monte-Carlo approach over the exact method. The Monte-Carlo approach was able to meet or surpass the value of the exact method in almost all cases. In the largest examples, the Monte-Carlo method had a much better performance: in \emph{CoralLib/linderoth, Knapsack\_6} and \emph{Knapsack\_7}, it found the optimal solution much faster than the exact method. Moreover, in \emph{Knapsack\_8} and \emph{Knapsack\_9}, the Monte-Carlo algorithm found much better solutions than the exact method in shorter running times.

Note that this qualitative analysis for the Monte-Carlo algorithm is largely possible because we have the exact solution to compare with in most cases. In practice, if the enumeration algorithm cannot be executed, we would need to rely on the volume-errors. In this line, the results are less positive. While most instances of BOLib and CoralLib have a volume-error of $0\%$, some instances have a large error, even attaining $74\%$ in the worst one, which happens to be the best one in time (\emph{CoralLib/linderoth}). Probably, these instances have many small full-dimensional chambers, and so the volume-error estimator does not reflect how well the algorithm is performing (it largely overestimates the error).

We also note that, in theory, the Knapsack problem should have a large volume-error (since it has as many full-dimensional chambers as chamber vertices, as illustrated in Section \ref{sec:Examples}). Nevertheless, the $0\%$ error estimation we obtained tells us that most of these chambers are, in fact, small. This could be exploited by a more ad-hoc sampling technique (instead of just plain Monte-Carlo).

Moving forward, the main (and clear) challenge for this work is scalability. These results show short running times since all instances have small dimensions. The main bottleneck currently is the enumeration of the faces of a polytope. In the case of Algorithm \ref{alg:Enumeration}, there does not seem to be much hope in improving this substantially: note in Table \ref{table:mainresults} that in all but the bottom two entries\footnote{The last two entries of Table \ref{table:mainresults} correspond to cases where $\mathscr{F}_{\leq \nx}$ was not fully computed due to the time limit.}, Algorithm \ref{alg:Enumeration} used all available faces. This is because the algorithm heavily relies on \emph{maximal labels}, which is important in our procedure to not repeat chambers when enumerating. Nonetheless, we still believe Algorithm \ref{alg:Enumeration} can be useful as a baseline that has optimality guarantees. \DS{As an immediate perspective, we propose to conduct a study to understand the computational complexity of the problems presented in this work.}

Algorithm \ref{alg:MonteCarlo} could potentially be improved significantly. First of all, recall that this approach only uses the faces of dimension $\nx$ (i.e. $\mathscr{F}_{\nx}$), which can be considerably smaller than $\mathscr{F}_{\leq \nx}$ (see columns 3 and 4 of Table \ref{table:mainresults}). Therefore, a more intricate enumeration that exploits this could be devised. Additionally, and perhaps more importantly, note that in the instances where $|\mathscr{F}_{\nx}|$ is not too small (say, more than 40) Algorithm \ref{alg:MonteCarlo} only uses a fraction of $\mathscr{F}_{\nx}$ in its execution. This indicates that one could heavily restrict the faces to consider initially and generate more on-the-fly, much like in a column generation approach. Another potential improvement path is exploiting more structure of a particular family of instances (as in the knapsack case), which may indicate which are the faces that one would truly need.

\begin{landscape}
\begin{table}[t]
		\centering

\caption{\small Summary of results for Algorithms \ref{alg:Enumeration} and \ref{alg:MonteCarlo} for selected  BOLib instances \cite{zhou2020bolib},  CoralLib instances \cite{corallib} and Knapsack instances \cite{BuchheimHenkeIrmai2022Knapsack}.  The ``Size'' of the instance is $(\nx + \ny, m)$. The  ``Obj gap'' column shows the gap between the values found for both algorithms; a negative gap indicates the stochastic method performed better. The ``Error" column shows the upper estimation of the \emph{volume of unseen chambers} during the sampling process as per \eqref{eq-def:DoubleEstimator}. The columns labeled ``Computation Times'' contain the running times (in seconds) for the computation of all the faces, the execution of Algorithm \ref{alg:Enumeration} and of Algorithm \ref{alg:MonteCarlo}. The columns labeled ``Used faces'' contain the number of faces that were explicitly used during the execution of each algorithm.}\label{table:mainresults}

		\begin{adjustbox}{scale=0.84}
			\begin{tabular}{lrrrrr|rrr|rr|}
				\cline{7-11}
				& \multicolumn{1}{c}{}         & \multicolumn{1}{c}{}       & \multicolumn{1}{l}{}          & \multicolumn{1}{c}{}        & \multicolumn{1}{c|}{}      & \multicolumn{3}{c|}{Computation Times}                                                 & \multicolumn{2}{c|}{Used Faces}                           \\ \hline
				\multicolumn{1}{|l|}{Instance}                      & \multicolumn{1}{c|}{Size}    & \multicolumn{1}{c|}{$|\mathscr{F}_{\leq \nx}|$} & \multicolumn{1}{l|}{$|\mathscr{F}_{\nx}|$} & \multicolumn{1}{c|}{Obj gap}    & \multicolumn{1}{c|}{Error} & \multicolumn{1}{c|}{Faces} & \multicolumn{1}{c|}{Alg. 1} & \multicolumn{1}{c|}{Alg. 2} & \multicolumn{1}{c|}{Alg. 1} & \multicolumn{1}{c|}{Alg 2.} \\ \hline
				\multicolumn{1}{|l|}{BOLib/AnandalinghamWhite1990} & \multicolumn{1}{r|}{(2,7)}   & \multicolumn{1}{r|}{12}    & \multicolumn{1}{r|}{6}        & \multicolumn{1}{r|}{0\%}    & 0\%                        & \multicolumn{1}{r|}{1.4}   & \multicolumn{1}{r|}{3.8}    & 7.2                         & \multicolumn{1}{r|}{12}     & 5                           \\ \hline
				\multicolumn{1}{|l|}{BOLib/Bard1984a}              & \multicolumn{1}{r|}{(2,6)}   & \multicolumn{1}{r|}{10}    & \multicolumn{1}{r|}{5}        & \multicolumn{1}{r|}{0\%}    & 0\%                        & \multicolumn{1}{r|}{3.1}   & \multicolumn{1}{r|}{6.8}    & 7.9                         & \multicolumn{1}{r|}{10}     & 5                           \\ \hline
				\multicolumn{1}{|l|}{BOLib/Bard1984b}              & \multicolumn{1}{r|}{(2,6)}   & \multicolumn{1}{r|}{10}    & \multicolumn{1}{r|}{5}        & \multicolumn{1}{r|}{0\%}    & 0\%                        & \multicolumn{1}{r|}{1.4}   & \multicolumn{1}{r|}{4.1}    & 7.5                         & \multicolumn{1}{r|}{10}     & 5                           \\ \hline
				\multicolumn{1}{|l|}{BOLib/Bard1991Ex2}            & \multicolumn{1}{r|}{(3,6)}   & \multicolumn{1}{r|}{14}    & \multicolumn{1}{r|}{9}        & \multicolumn{1}{r|}{0\%}    & 0\%                        & \multicolumn{1}{r|}{1.4}   & \multicolumn{1}{r|}{4.2}    & 8.2                         & \multicolumn{1}{r|}{14}     & 6                           \\ \hline
				\multicolumn{1}{|l|}{BOLib/BardFalk1982Ex2}        & \multicolumn{1}{r|}{(4,7)}   & \multicolumn{1}{r|}{45}    & \multicolumn{1}{r|}{17}       & \multicolumn{1}{r|}{0\%}    & 45\%                       & \multicolumn{1}{r|}{1.7}   & \multicolumn{1}{r|}{4.3}    & 7.6                         & \multicolumn{1}{r|}{45}     & 5                           \\ \hline
				\multicolumn{1}{|l|}{BOLib/BenAyedBlair1990a}      & \multicolumn{1}{r|}{(3,6)}   & \multicolumn{1}{r|}{20}    & \multicolumn{1}{r|}{12}       & \multicolumn{1}{r|}{0\%}    & 0\%                        & \multicolumn{1}{r|}{1.4}   & \multicolumn{1}{r|}{4.2}    & 8.5                         & \multicolumn{1}{r|}{20}     & 4                           \\ \hline
				\multicolumn{1}{|l|}{BOLib/BenAyedBlair1990b}      & \multicolumn{1}{r|}{(2,5)}   & \multicolumn{1}{r|}{6}     & \multicolumn{1}{r|}{3}        & \multicolumn{1}{r|}{0\%}    & 0\%                        & \multicolumn{1}{r|}{1.5}   & \multicolumn{1}{r|}{4.3}    & 8.0                         & \multicolumn{1}{r|}{6}      & 3                           \\ \hline
				\multicolumn{1}{|l|}{BOLib/BialasKarwan1984a}      & \multicolumn{1}{r|}{(3,8)}   & \multicolumn{1}{r|}{20}    & \multicolumn{1}{r|}{12}       & \multicolumn{1}{r|}{0\%}    & 0\%                        & \multicolumn{1}{r|}{1.6}   & \multicolumn{1}{r|}{4.4}    & 8.4                         & \multicolumn{1}{r|}{20}     & 10                          \\ \hline
				\multicolumn{1}{|l|}{BOLib/BialasKarwan1984b}      & \multicolumn{1}{r|}{(2,7)}   & \multicolumn{1}{r|}{12}    & \multicolumn{1}{r|}{6}        & \multicolumn{1}{r|}{0\%}    & 0\%                        & \multicolumn{1}{r|}{1.5}   & \multicolumn{1}{r|}{4.3}    & 8.0                         & \multicolumn{1}{r|}{12}     & 5                           \\ \hline
				\multicolumn{1}{|l|}{BOLib/CandlerTownsley1982}    & \multicolumn{1}{r|}{(5,8)}   & \multicolumn{1}{r|}{111}   & \multicolumn{1}{r|}{48}       & \multicolumn{1}{r|}{1\%}    & 37\%                       & \multicolumn{1}{r|}{1.8}   & \multicolumn{1}{r|}{7.9}    & 18.5                        & \multicolumn{1}{r|}{111}    & 16                          \\ \hline
				\multicolumn{1}{|l|}{BOLib/ClarkWesterberg1988}    & \multicolumn{1}{r|}{(2,3)}   & \multicolumn{1}{r|}{6}     & \multicolumn{1}{r|}{3}        & \multicolumn{1}{r|}{0\%}    & 0\%                        & \multicolumn{1}{r|}{1.5}   & \multicolumn{1}{r|}{4.3}    & 8.0                         & \multicolumn{1}{r|}{6}      & 3                           \\ \hline
				\multicolumn{1}{|l|}{BOLib/ClarkWesterberg1990b}   & \multicolumn{1}{r|}{(3,7)}   & \multicolumn{1}{r|}{15}    & \multicolumn{1}{r|}{9}        & \multicolumn{1}{r|}{0\%}    & 0\%                        & \multicolumn{1}{r|}{1.5}   & \multicolumn{1}{r|}{4.4}    & 8.3                         & \multicolumn{1}{r|}{15}     & 9                           \\ \hline
				\multicolumn{1}{|l|}{BOLib/GlackinEtal2009}        & \multicolumn{1}{r|}{(3,6)}   & \multicolumn{1}{r|}{20}    & \multicolumn{1}{r|}{5}        & \multicolumn{1}{r|}{0\%}    & 48\%                       & \multicolumn{1}{r|}{1.6}   & \multicolumn{1}{r|}{4.8}    & 8.0                         & \multicolumn{1}{r|}{20}     & 3                           \\ \hline
				\multicolumn{1}{|l|}{BOLib/HaurieSavardWhite1990}  & \multicolumn{1}{r|}{(2,4)}   & \multicolumn{1}{r|}{8}     & \multicolumn{1}{r|}{4}        & \multicolumn{1}{r|}{0\%}    & 0\%                        & \multicolumn{1}{r|}{1.6}   & \multicolumn{1}{r|}{4.5}    & 8.3                         & \multicolumn{1}{r|}{8}      & 4                           \\ \hline
				\multicolumn{1}{|l|}{BOLib/HuHuangZhang2009}       & \multicolumn{1}{r|}{(3,6)}   & \multicolumn{1}{r|}{20}    & \multicolumn{1}{r|}{12}       & \multicolumn{1}{r|}{0\%}    & 0\%                        & \multicolumn{1}{r|}{1.6}   & \multicolumn{1}{r|}{4.4}    & 8.6                         & \multicolumn{1}{r|}{20}     & 7                           \\ \hline
				\multicolumn{1}{|l|}{BOLib/LanWenShihLee2007}      & \multicolumn{1}{r|}{(2,8)}   & \multicolumn{1}{r|}{14}    & \multicolumn{1}{r|}{7}        & \multicolumn{1}{r|}{0\%}    & 1\%                        & \multicolumn{1}{r|}{1.6}   & \multicolumn{1}{r|}{4.5}    & 8.2                         & \multicolumn{1}{r|}{14}     & 6                           \\ \hline
				\multicolumn{1}{|l|}{BOLib/LiuHart1994}            & \multicolumn{1}{r|}{(2,5)}   & \multicolumn{1}{r|}{10}    & \multicolumn{1}{r|}{5}        & \multicolumn{1}{r|}{0\%}    & 0\%                        & \multicolumn{1}{r|}{1.6}   & \multicolumn{1}{r|}{4.4}    & 8.6                         & \multicolumn{1}{r|}{10}     & 4                           \\ \hline
				\multicolumn{1}{|l|}{BOLib/MershaDempe2006Ex1}     & \multicolumn{1}{r|}{(2,6)}   & \multicolumn{1}{r|}{8}     & \multicolumn{1}{r|}{4}        & \multicolumn{1}{r|}{0\%}    & 0\%                        & \multicolumn{1}{r|}{1.7}   & \multicolumn{1}{r|}{4.8}    & 8.6                         & \multicolumn{1}{r|}{8}      & 4                           \\ \hline
				\multicolumn{1}{|l|}{BOLib/MershaDempe2006Ex2}     & \multicolumn{1}{r|}{(2,7)}   & \multicolumn{1}{r|}{10}    & \multicolumn{1}{r|}{5}        & \multicolumn{1}{r|}{0\%}    & 64\%                       & \multicolumn{1}{r|}{1.5}   & \multicolumn{1}{r|}{4.3}    & 6.4                         & \multicolumn{1}{r|}{10}     & 3                           \\ \hline
				\multicolumn{1}{|l|}{BOLib/TuyEtal1993}            & \multicolumn{1}{r|}{(4,7)}   & \multicolumn{1}{r|}{45}    & \multicolumn{1}{r|}{17}       & \multicolumn{1}{r|}{0\%}    & 54\%                       & \multicolumn{1}{r|}{1.7}   & \multicolumn{1}{r|}{4.5}    & 7.6                         & \multicolumn{1}{r|}{45}     & 5                           \\ \hline
				\multicolumn{1}{|l|}{BOLib/TuyEtal1994}            & \multicolumn{1}{r|}{(4,8)}   & \multicolumn{1}{r|}{72}    & \multicolumn{1}{r|}{24}       & \multicolumn{1}{r|}{0\%}    & 46\%                       & \multicolumn{1}{r|}{1.6}   & \multicolumn{1}{r|}{4.5}    & 7.5                         & \multicolumn{1}{r|}{72}     & 6                           \\ \hline
				\multicolumn{1}{|l|}{BOLib/VisweswaranEtal1996}    & \multicolumn{1}{r|}{(2,6)}   & \multicolumn{1}{r|}{8}     & \multicolumn{1}{r|}{4}        & \multicolumn{1}{r|}{0\%}    & 0\%                        & \multicolumn{1}{r|}{1.4}   & \multicolumn{1}{r|}{4.3}    & 7.9                         & \multicolumn{1}{r|}{8}      & 4                           \\ \hline
				\multicolumn{1}{|l|}{BOLib/WangJiaoLi2005}         & \multicolumn{1}{r|}{(3,7)}   & \multicolumn{1}{r|}{23}    & \multicolumn{1}{r|}{14}       & \multicolumn{1}{r|}{0\%}    & 0\%                        & \multicolumn{1}{r|}{1.5}   & \multicolumn{1}{r|}{4.4}    & 8.6                         & \multicolumn{1}{r|}{23}     & 5                           \\ \hline \hline
				\multicolumn{1}{|l|}{CoralLib/linderoth}          & \multicolumn{1}{r|}{(6,15)}  & \multicolumn{1}{r|}{545}   & \multicolumn{1}{r|}{51}       & \multicolumn{1}{r|}{0\%}    & 74\%                       & \multicolumn{1}{r|}{1.4}   & \multicolumn{1}{r|}{148.0}  & 7.4                         & \multicolumn{1}{r|}{545}    & 7                           \\ \hline
				\multicolumn{1}{|l|}{CoralLib/moore90\_2}         & \multicolumn{1}{r|}{(2,7)}   & \multicolumn{1}{r|}{12}    & \multicolumn{1}{r|}{6}        & \multicolumn{1}{r|}{0\%}    & 0\%                        & \multicolumn{1}{r|}{1.4}   & \multicolumn{1}{r|}{4.0}    & 7.6                         & \multicolumn{1}{r|}{12}     & 5                           \\ \hline
				\multicolumn{1}{|l|}{CoralLib/moore90}            & \multicolumn{1}{r|}{(2,8)}   & \multicolumn{1}{r|}{8}     & \multicolumn{1}{r|}{4}        & \multicolumn{1}{r|}{0\%}    & 0\%                        & \multicolumn{1}{r|}{1.6}   & \multicolumn{1}{r|}{4.4}    & 8.1                         & \multicolumn{1}{r|}{8}      & 4                           \\ \hline \hline
				\multicolumn{1}{|l|}{Knapsack\_6}                     & \multicolumn{1}{r|}{(7,15)}  & \multicolumn{1}{r|}{574}   & \multicolumn{1}{r|}{447}      & \multicolumn{1}{r|}{0\%}    & 0\%                        & \multicolumn{1}{r|}{3.2}   & \multicolumn{1}{r|}{117.1}  & 19.2                        & \multicolumn{1}{r|}{574}    & 255                         \\ \hline
				\multicolumn{1}{|l|}{Knapsack\_7}                     & \multicolumn{1}{r|}{(8,17)}  & \multicolumn{1}{r|}{1278}  & \multicolumn{1}{r|}{1023}     & \multicolumn{1}{r|}{0\%}    & 0\%                        & \multicolumn{1}{r|}{2.5}   & \multicolumn{1}{r|}{626.9}  & 38.5                        & \multicolumn{1}{r|}{1278}   & 575                         \\ \hline
				\multicolumn{1}{|l|}{Knapsack\_8}                     & \multicolumn{1}{r|}{(9,19)}  & \multicolumn{1}{r|}{2814}  & \multicolumn{1}{r|}{2303}     & \multicolumn{1}{r|}{-14\%}  & 0\%                        & \multicolumn{1}{r|}{4.9}   & \multicolumn{1}{r|}{914.7}  & 82.4                        & \multicolumn{1}{r|}{1990}   & 1279                        \\ \hline
				\multicolumn{1}{|l|}{Knapsack\_9}                     & \multicolumn{1}{r|}{(10,21)} & \multicolumn{1}{r|}{6142}  & \multicolumn{1}{r|}{5119}     & \multicolumn{1}{r|}{-380\%} & 0\%                        & \multicolumn{1}{r|}{15.4}  & \multicolumn{1}{r|}{984.2}  & 181.4                       & \multicolumn{1}{r|}{3346}   & 2815                        \\ \hline
			\end{tabular}
		\end{adjustbox}
  \end{table}
\end{landscape}


%

\section*{Statements and Declaration.}

\noindent{\bf Funding.} G. Mu\~noz was supported by FONDECYT Iniciaci\'on 11190515 (ANID-Chile). D. Salas was supported by the Center of Mathematical Modeling (CMM) FB210005 BASAL funds for centers of excellence (ANID-Chile), and the grant FONDECYT Iniciaci\'on 11220586 (ANID-Chile). A. Svensson was supported by the grant FONDECYT postdoctorado 3210735 (ANID-Chile).
%

\bibliographystyle{plain}
\bibliography{Biblio3.bib}

\end{document}